\documentclass[a4paper,onecolumn, superscriptaddress,10pt,
accepted=2020-05-14,issue=3, volume=2,
shorttitle=papers/compositionality-2-3]{compositionalityarticle}
\pdfoutput=1
\usepackage{eigilscmds}
\usepackage{tikzit}
\usepackage{geometry}
\usepackage{enumerate}
\input{comonoids.tikzdefs}
\tikzstyle{morphism}=[fill=white, draw=black, shape=rectangle]
\tikzstyle{medium box}=[fill=white, draw=black, shape=rectangle, minimum width=0.8cm, minimum height=0.9cm]
\tikzstyle{large morphism}=[fill=white, draw=black, shape=rectangle, minimum width=1.7cm, minimum height=1cm]
\tikzstyle{bn}=[fill=black, draw=black, shape=circle, inner sep=1.5pt]
\tikzstyle{state}=[fill=white, draw=black, regular polygon, regular polygon sides=3, minimum width=0.8cm, shape border rotate=180, inner sep=0pt]
\tikzstyle{medium state}=[fill=white, draw=black, regular polygon, regular polygon sides=3, minimum width=1.3cm, inner sep=0pt, shape border rotate=180]
\tikzstyle{large state}=[fill=white, draw=black, regular polygon, regular polygon sides=3, minimum width=2.2cm, shape border rotate=180, inner sep=0pt]
\tikzstyle{wn}=[fill=white, draw=black, shape=circle, inner sep=1.5pt]
\tikzstyle{large brace}=[draw, decorate, decoration={{brace,amplitude=12pt}}, inner sep=12pt, rotate=-90, xshift=34pt]

\tikzstyle{arrow}=[->]
\tikzstyle{dashed line}=[-, dashed]
\tikzstyle{double line}=[-, double]

\author[Fritz]{Tobias Fritz}
\email{tfritz@pitp.ca}
\address{Perimeter Institute for Theoretical Physics, N2L 2Y5, Waterloo, Ontario, Canada}
\author[Rischel]{Eigil Fjeldgren Rischel}
\email{ayegill@gmail.com}
\address{Deparment of Mathematical Sciences, University of Copenhagen, 2100 Denmark}

\title{Infinite products and zero--one laws in categorical probability}
\date{2020-07-09}

\renewcommand{\sf}{\mathsf}

\begin{document}
\maketitle

\begin{abstract}
\emph{Markov categories} are a recent category-theoretic approach to the foundations of probability and statistics. Here we develop this approach further by treating infinite products and the Kolmogorov extension theorem. This is relevant for all aspects of probability theory in which infinitely many random variables appear at a time. These infinite tensor products $\bigotimes_{i \in J} X_i$ come in two versions: a weaker but more general one for families of objects $(X_i)_{i \in J}$ in semicartesian symmetric monoidal categories, and a stronger but more specific one for families of objects in Markov categories.

	As a first application, we state and prove versions of the zero--one laws of Kolmogorov and Hewitt--Savage for Markov categories. This gives general versions of these results which can be instantiated not only in measure-theoretic probability, where they specialize to the standard ones in the setting of standard Borel spaces, but also in other contexts. % For example, applying the Kolmogorov zero--one law to the Kleisli category of the hyperspace monad on topological spaces gives a criterion for when maps out of an infinite product of topological spaces into a Hausdorff space are constant.

\end{abstract}

\tableofcontents

\section{Introduction}

\emph{Markov categories} are an approach to the foundations of probability and statistics based on category theory, proposed first by Golubtsov~\cite{golubtsov}, rediscovered recently independently by Cho and Jacobs~\cite{cho_jacobs}, and developed extensively by the first-named author in \cite{markov_cats}. The basic observation is that Markov kernels can be composed sequentially and in parallel, making them into a symmetric monoidal category $\Stoch$. One can then find axioms which make a symmetric monoidal category behave similarly to the actual category of Markov kernels, and then state and prove results from probability and statistics in this general abstract context. Thus Markov categories are abstract versions of the category of Markov kernels, and hence their name.

This represents a \emph{synthetic} approach to probability, in contrast with the usual measure-theoretic approach (which can be called \emph{analytic}). Instead of defining \emph{what} the words probability, distribution, random variable, etc, mean semantically, we instead state by which operations these objects can be combined and related to each other, and which properties they are supposed to satisfy with respect to each other.
We can draw an analogy to the Euclidean approach to geometry: in this synthetic approach, points and lines are described only in terms of their relationships with one another. This is to be contrasted with the analytic approach often attributed to Descartes, in which geometric objects are defined concretely in terms of more primitive notions like sets and real numbers.

In this paper, we give a treatment of further fundamental and classical results of probability theory within the setting of Markov categories. This adds further support to the hypothesis that Markov categories are an adequate setting for the foundations of probability and statistics, adding to the existing treatments of Bayesian updating, almost sure equality, conditional independence, and theorems on sufficient statistics~\cite{cho_jacobs,markov_cats}. Concretely, we develop a notion of infinite tensor products in Markov categories and then apply this notion to state and prove two of the standard zero--one laws of probability theory.

\paragraph*{Summary.}

We now present a more detailed overview.

\Cref{background} presents some background by recalling the definition of Markov category, presenting some pertinent examples, and discussing those aspects of the existing theory of Markov categories that will be of relevance for the rest of the paper. With the exception of \Cref{cring_not_causal}, there is nothing new in this section.

\Cref{infprod_semicartesian} then proceeds by introducing infinite tensor products in semicartesian symmetric monoidal categories, which are exactly those symmetric monoidal categories in which the monoidal unit object is terminal. We consider a number of examples of such categories and discuss whether infinite tensor products exist. This already includes examples relevant for probability theory, and in particular $\BorelStoch$, the category of standard Borel spaces and measurable maps\footnote{Or equivalently Polish spaces and measurable maps.}, for which we explain how countably infinite tensor products implement the Kolmogorov extension theorem (for countably many variables); the extension theorem states that the joint distributions of infinitely many random variables are in bijection with compatible families of joint distributions for each finite subset of variables. The universal property of our infinite tensor products makes the extension theorem into a definition, generalized from distributions to Markov kernels. An important subtlety here is the additional preservation condition in \Cref{semicartesian_infproduct}, where infinite tensor products are introduced: this additional condition guarantees that infinite tensor products have nice compositionality properties, so that e.g.~a tensor product of two infinite tensor products is itself an infinite tensor product (\Cref{two_infproducts}). This condition does not come up traditionally in the context of the Kolmogorov extension theorem, the reason being that it automatically holds in the case of $\BorelStoch$ (\Cref{infprods_stoch}).

In the short \Cref{infprod_markov}, we then turn to infinite tensor products in Markov categories, which we call \emph{Kolmogorov products} due to the connection with the Kolmogorov extension theorem. In addition to being an infinite tensor product, a Kolmogorov product also needs to respect the additional structure which makes a semicartesian symmetric monoidal category into a Markov category, namely the comultiplications $\cop{X} : X \to X \otimes X$. This is implemented by postulating that the product projections should be deterministic morphisms, or equivalently that the infinite product should be an honest categorical product in the cartesian monoidal subcategory of deterministic morphisms (\Cref{kolmprod_catdet}). Not surprisingly, the infinite tensor products in $\BorelStoch$, which implement the Kolmogorov extension theorem, are indeed Kolmogorov products.

In \Cref{zeroones}, we state and prove the zero--one laws of Kolmogorov and Hewitt--Savage in their synthetic versions. Very roughly, the Kolmogorov zero--one law states that, given an infinite collection of independent random variables and an event which is in the $\sigma$-algebra generated by the variables but independent of any finite subset of them, the event has probability zero or one. This theorem goes back to Kolmogorov's foundational monograph \cite{kolmogorovgerm}, although the statement appearing there is significantly different from its modern form. For a more recent textbook treatment of this theorem, see e.g. \cite[Theorem~2.37]{klenke}

The Hewitt--Savage zero--one law, due to Hewitt and Savage in \cite[Theorem~11.3]{hewitt_savage}, is similar in flavour but technically quite different. Here, the infinitely many variables are assumed to be identically distributed in addition to being independent, and the event is assumed to be invariant under finite permutations of the variables. Again, under these assumptions, the event has probability zero or one.

Finally in \Cref{examples}, we apply our results to $\BorelStoch$. We thereby rederive the standard zero--one laws for random variables in standard Borel spaces from our abstract theorems. We also instantiate one of our theorems in the Kleisli category of the lower Vietoris monad, which gives criteria for when a continuous map out of infinite product of topological spaces into a Hausdorff space is constant (\Cref{vietoris_kolmogorov}).

\paragraph*{Notation.} 

Throughout, $\cC$ is a symmetric monoidal category. To simplify notation, we generally omit mention of the structure isomorphisms by assuming that $\cC$ is strict without loss of generality. We routinely make use of string diagram notation, and in doing so we omit object labels whenever these are obvious from the context. Our string diagrams are to be read from bottom to top.

In \Cref{infprod_semicartesian}, $\cC$ denotes more concretely a semicartesian strict symmetric monoidal category. In \Cref{infprod_markov} and after, $\cC$ denotes a Markov category in the sense of \Cref{markov_cat}. $J$ is a set used for indexing products of objects; the definitions and results of this paper are nontrivial only when $J$ is infinite. $F$ denotes either an arbitrary finite set, or more concretely a finite subset of $J$.

\paragraph*{Acknowledgments.} We thank Paolo Perrone for discussions and \Cref{poset_case} as well as Kenta Cho for comments on an earlier version of this paper.

\section{Background on Markov categories}\label{background}

We assume familiarity with symmetric monoidal categories, up to and including string diagram notation and string diagram calculus for composite morphisms in symmetric monoidal categories.

A symmetric monoidal category $\cC$ is \emph{semicartesian} if the monoidal unit $I \in \cC$ is terminal. Equivalently~\cite[Theorem~3.5]{GLS}, $\cC$ is semicartesian if it comes equipped with morphisms
\begin{equation}
	\label{projections}
	X \otimes Y \longrightarrow X, \qquad X \otimes Y \longrightarrow Y,
\end{equation}
which are natural in $X$ and $Y$ and coincide with the monoidal structure isomorphisms whenever $X = I$ or $Y = I$. In the probability context, these morphisms can be interpreted as marginalizations, or equivalently as discarding the value of the variable $Y$ or $X$.

\begin{example}
  Any Cartesian monoidal category is obviously semicartesian. In the probabilistic setting that we are considering, this corresponds to a situation in which there is no randomness
  (in such a category, all morphisms are deterministic, cf. \cref{determ_def}). Hence it is somewhat uninteresting from our point of view.
\end{example}

Here are the main examples of semicartesian symmetric monoidal categories that we will be considering.

\begin{example}
	\label{finstoch}
	$\FinStoch$ is the category of finite sets with \emph{stochastic matrices} as morphisms. This means that for finite sets $X$ and $Y$, a morphism $f : X \to Y$ is a matrix $(f_{xy})_{x \in X,y \in Y}$ of nonnegative real numbers such that $\sum_y f_{xy} = 1$ for every $x$. Stochastic matrices compose by matrix multiplication, which in this context is also known as the \emph{Chapman--Kolmogorov equation}. We consider $\FinStoch$ as a symmetric monoidal category with respect to the cartesian product of sets on objects and the tensor product (Kronecker product) of stochastic matrices on morphisms; the monoidal structure isomorphisms are the obvious ones coming from the embedding $\FinSet \subseteq \FinStoch$.
\end{example}

\begin{example}
	\label{cring}
	Let $\CRing_+$ be the category of commutative rings with maps which are merely additive and unit-preserving, considered as a symmetric monoidal category in the obvious way with respect to the tensor product of rings. Then since $\CRing_+$ has the monoidal unit $\Z$ as its initial object, we can conclude that $\CRing_+\op$ is semicartesian monoidal.
\end{example}

\begin{example}
	\label{stoch}
	$\Stoch$ is the category of measurable spaces $(X,\Sigma_X)$ as objects and Markov kernels as morphisms. We sketch the definition here and refer to~\cite[Section~4]{markov_cats} and references therein for the details. Given measurable spaces $(X,\Sigma_X)$ and $(Y,\Sigma_Y)$, a \emph{Markov kernel} $(X,\Sigma_X) \to (Y,\Sigma_Y)$ is a map
	\[
		f \: : \: \Sigma_Y \times X \longrightarrow [0,1], \quad (S,x) \longmapsto f(S|x)
	\]
	assigning to every $x \in X$ a probability measure $f(-|x) : \Sigma_Y \to [0,1]$ in such a way that for every $S \in \Sigma_Y$, the function $x \mapsto f(S|x)$ is measurable. Intuitively, one can think of $f$ as assigning to every $x \in X$ a random element of $Y$; or one can think of $f$ as a statistical model, specifying one probability measure $f(-|x)$ for every parameter value $x$.
	
	Composition of Markov kernels is again defined by a variant of the Chapman--Kolmogorov equation: for $f : (X,\Sigma_X) \to (Y,\Sigma_Y)$ and $g : (Y,\Sigma_Y) \to (Z,\Sigma_Z)$, the composite is given by, for all $x \in X$ and $T \in \Sigma_Z$,
	\[
		(gf)(T|x) := \int_{y \in Y} g(T|y) \, f(dy|x),
	\]
	where the relevant measurability condition in order for this to define a Markov kernel follows by a standard application of the $\pi$-$\lambda$-theorem; associativity of composition uses Fubini's theorem. In this way, we obtain a category $\Stoch$ which is, more or less by definition, isomorphic to the Kleisli category of the Giry monad~\cite{giry}.

	The symmetric monoidal structure on $\Stoch$ is defined in terms of the usual product of measurable spaces, given by the Cartesian product of underlying sets equipped with the product $\sigma$-algebra. The tensor product of morphisms $f : (A,\Sigma_A) \to (X,\Sigma_X)$ and $g : (B,\Sigma_B) \to (Y,\Sigma_Y)$ is given by the unique Markov kernel between the product measurable spaces which satisfies
	\[
		(f \otimes g)(S \times T|a,b) = f(S|a) \, g(T|b) \qquad \forall a \in A,\: b \in B,\: S \in \Sigma_X,\: T \in \Sigma_Y.
	\]
	The monoidal unit $I$ is given by any one-element set with its unique $\sigma$-algebra.
    As a special case, morphisms $I \to (X,\Sigma_X)$ can be identified with probability measures $\Sigma_X \to [0,1]$.
    \hyphenation{sym-me-tric}
   
	There is a well-known subclass of particularly well-behaved measurable spaces, the \emph{standard Borel spaces}, or equivalently Polish spaces equipped with their Borel $\sigma$-algebras. We write\linebreak $\BorelStoch \subseteq\nobreak \Stoch$ for the full subcategory of standard Borel spaces with Markov kernels. Since finite products of standard Borel spaces are again standard Borel, $\BorelStoch$ is again semicartesian symmetric monoidal. In summary, $\BorelStoch$ is the category of standard Borel spaces as objects with measurable Markov kernels as morphisms.
\end{example}

Returning to the general theory, in terms of string diagrams the unique morphism $\discard{X} : X \to I$ for an object $X \in \cC$ is denoted by
\[
	\begin{tikzpicture}
	\begin{pgfonlayer}{nodelayer}
		\node [style=bn] (0) at (0, 0.5) {};
		\node [style=none] (1) at (0, -0.5) {};
		\node [style=none] (2) at (0, -1) {$X$};
	\end{pgfonlayer}
	\begin{pgfonlayer}{edgelayer}
		\draw (0) to (1.center);
	\end{pgfonlayer}
\end{tikzpicture}

\]
The projection maps $\id \otimes \discard{Y} : X \otimes Y \to X$ and $\discard{X} \otimes \id : X \otimes Y \to Y$, which are the ones from~\eqref{projections}, are correspondingly written as
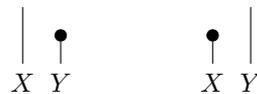
\[
	\begin{tikzpicture}
	\begin{pgfonlayer}{nodelayer}
		\node [style=none] (0) at (-2, -0.5) {};
		\node [style=none] (1) at (-3, 1) {};
		\node [style=bn] (2) at (-2, 0.25) {};
		\node [style=none] (3) at (-3, -0.5) {};
		\node [style=none] (4) at (-2, -1) {$Y$};
		\node [style=none] (5) at (-3, -1) {$X$};
		\node [style=none] (6) at (3, -0.5) {};
		\node [style=bn] (7) at (2, 0.25) {};
		\node [style=none] (8) at (3, 1) {};
		\node [style=none] (9) at (2, -0.5) {};
		\node [style=none] (10) at (3, -1) {$Y$};
		\node [style=none] (11) at (2, -1) {$X$};
	\end{pgfonlayer}
	\begin{pgfonlayer}{edgelayer}
		\draw (1.center) to (3.center);
		\draw (2) to (0.center);
		\draw (7) to (9.center);
		\draw (8.center) to (6.center);
	\end{pgfonlayer}
\end{tikzpicture}
\]

Our goal is to use semicartesian monoidal categories in order to develop aspects of probability theory in categorical terms, in such a way that instantiating this theory in $\Stoch$ or $\BorelStoch$ recovers the standard theory. As it turns out, doing so requires a bit more structure, in a form which has been axiomatized first by Cho and Jacobs~\cite{cho_jacobs} as \emph{affine CD-categories}, although very similar definitions occur in earlier work of Golubtsov~\cite{golubtsov}. We here follow the more intuitive terminology of our own~\cite[Definition~2.1]{markov_cats}. 

\begin{definition}
	A \emph{Markov category} $\cC$ is a semicartesian symmetric monoidal category where every object $X \in \cC$ is equipped with a distinguished morphism
	\begin{equation}
		\label{comultiplication}
		\begin{tikzpicture}
	\begin{pgfonlayer}{nodelayer}
		\node [style=bn] (0) at (0, 0) {};
		\node [style=none] (1) at (-1, 1) {};
		\node [style=none] (2) at (0, -1) {};
		\node [style=none] (3) at (1, 1) {};
		\node [style=none] (4) at (-1, 1.5) {$X$};
		\node [style=none] (6) at (1, 1.5) {$X$};
		\node [style=none] (7) at (0, -1.5) {$X$};
		\node [style=none] (8) at (-4, 0) {$\cop{X} \quad =$};
	\end{pgfonlayer}
	\begin{pgfonlayer}{edgelayer}
		\draw [bend right=45] (1.center) to (0);
		\draw [bend right=45] (0) to (3.center);
		\draw (0) to (2.center);
	\end{pgfonlayer}
\end{tikzpicture}
	\end{equation}
	which, together with $\discard{X} : X \to I$, makes $X$ into a commutative comonoid, and such that
	\begin{equation}
		\label{multiplicativity}
		\begin{tikzpicture}
	\begin{pgfonlayer}{nodelayer}
		\node [style=none] (0) at (-5, 2.5) {$X\otimes Y$};
		\node [style=none] (1) at (3, -2) {};
		\node [style=bn] (2) at (3, -1) {};
		\node [style=none] (3) at (4, 0) {};
		\node [style=none] (4) at (4, 0) {};
		\node [style=none] (5) at (2, 0) {};
		\node [style=none] (6) at (7, -2) {};
		\node [style=bn] (7) at (7, -1) {};
		\node [style=none] (8) at (8, 0) {};
		\node [style=none] (9) at (8, 0) {};
		\node [style=none] (10) at (6, 0) {};
		\node [style=none] (11) at (2, 2) {};
		\node [style=none] (12) at (4, 2) {};
		\node [style=none] (13) at (6, 2) {};
		\node [style=none] (14) at (8, 2) {};
		\node [style=none] (15) at (0, 0) {$=$};
		\node [style=bn] (16) at (-3.5, -0.5) {};
		\node [style=none] (17) at (-3.5, -2) {};
		\node [style=none] (18) at (-5, 1) {};
		\node [style=none] (19) at (-2, 1) {};
		\node [style=none] (20) at (-2, 2.5) {$X\otimes Y$};
		\node [style=none] (21) at (-3.5, -2.5) {$X\otimes Y$};
		\node [style=none] (22) at (-5, 2) {};
		\node [style=none] (23) at (-2, 2) {};
		\node [style=none] (24) at (2, 2.5) {$X$};
		\node [style=none] (25) at (4, 2.5) {$Y$};
		\node [style=none] (26) at (6, 2.5) {$X$};
		\node [style=none] (27) at (8, 2.5) {$Y$};
		\node [style=none] (28) at (3, -2.5) {$X$};
		\node [style=none] (29) at (7, -2.5) {$Y$};
	\end{pgfonlayer}
	\begin{pgfonlayer}{edgelayer}
		\draw [style=none] (1.center) to (2);
		\draw [style=none, bend left=45] (2) to (5.center);
		\draw [style=none, bend right=45] (2) to (4.center);
		\draw [style=none] (6.center) to (7);
		\draw [style=none, bend left=45] (7) to (10.center);
		\draw [style=none, bend right=45] (7) to (9.center);
		\draw [style=none] (11.center) to (5.center);
		\draw [style=none] (8.center) to (14.center);
		\draw [style=none, in=-90, out=90] (3.center) to (13.center);
		\draw [style=none, in=90, out=-90] (12.center) to (10.center);
		\draw (17.center) to (16);
		\draw [bend right=315] (16) to (18.center);
		\draw [bend right=45] (16) to (19.center);
		\draw (23.center) to (19.center);
		\draw (18.center) to (22.center);
	\end{pgfonlayer}
\end{tikzpicture}
	\end{equation}
	for all $X, Y \in \cC$.
	\label{markov_cat}
\end{definition}

There is a strictification result~\cite[Theorem~10.17]{markov_cats} which guarantees that $\cC$ can be assumed to be strict monoidal without messing up the commutative comonoid structures. We assume from now on that $\cC$ is strict throughout.

We think of the comultiplication~\Cref{comultiplication} as a copying operation. The coassociativity and counitality conditions guarantee that copying with any number of output wires is well-defined, and we draw it likewise as a single black dot with any number of outgoing wires, like this:
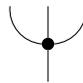
\[
	\begin{tikzpicture}
	\begin{pgfonlayer}{nodelayer}
		\node [style=bn] (0) at (0, 0) {};
		\node [style=none] (1) at (-1, 1) {};
		\node [style=none] (2) at (0, -1) {};
		\node [style=none] (3) at (0, 1) {};
		\node [style=none] (4) at (1, 1) {};
	\end{pgfonlayer}
	\begin{pgfonlayer}{edgelayer}
		\draw (3.center) to (0);
		\draw [bend left=45] (0) to (1.center);
		\draw (0) to (2.center);
		\draw [bend right=45] (0) to (4.center);
	\end{pgfonlayer}
\end{tikzpicture}
  \]

\Cref{finstoch,cring,stoch} are all Markov categories in a canonical way. In $\FinStoch$ and $\Stoch$, the comultiplication morphisms are given indeed by copying, i.e.~by those stochastic matrices or Markov kernels which map an element $x \in X$ to the Dirac delta measure at $(x,x) \in X \times X$. We refer to~\cite[Example~2.5 and Section~4]{markov_cats} for the technical details. In $\CRing_+\op$, we define the comultiplication of an object $R \in \CRing_+\op$ to be represented by the multiplication map $R \otimes R \to R$. The multiplicativity condition~\eqref{multiplicativity} then amounts to the fact that the multiplication on a tensor product of commutative rings $R \otimes S$ is given by the defining equation
\[
	(r_1 \otimes s_1) (r_2 \otimes s_2) \, = \, r_1 r_2 \otimes s_1 s_2.
\]

The following definition is among the central notions of the theory of Markov categories.

\begin{definition}[{\cite[Definition~10.1]{markov_cats}}]\label{determ_def}
	Let $\cC$ be a Markov category. A morphism $f : X \to Y$ in $\cC$ is \emph{deterministic} if it is a comonoid homomorphism,
	\[
		\begin{tikzpicture}
	\begin{pgfonlayer}{nodelayer}
		\node [style=none] (0) at (-3, -2) {};
		\node [style=bn] (1) at (-3, 0) {};
		\node [style=none] (2) at (-4, 1) {};
		\node [style=none] (3) at (-2, 1) {};
		\node [style=none] (4) at (-2, 1) {};
		\node [style=morphism] (5) at (-4, 1.25) {$f$};
		\node [style=morphism] (6) at (-2, 1.25) {$f$};
		\node [style=none] (7) at (3, -2) {};
		\node [style=bn] (8) at (3, 0) {};
		\node [style=none] (9) at (2, 1) {};
		\node [style=none] (10) at (4, 1) {};
		\node [style=none] (11) at (4, 1) {};
		\node [style=none] (12) at (-4, 2.25) {};
		\node [style=none] (13) at (-2, 2.25) {};
		\node [style=none] (14) at (2, 2.25) {};
		\node [style=none] (15) at (4, 2.25) {};
		\node [style=none] (16) at (0, 0) {$=$};
		\node [style=morphism] (17) at (3, -1) {$f$};
		\node [style=none] (18) at (-4, 2.75) {$Y$};
		\node [style=none] (19) at (-2, 2.75) {$Y$};
		\node [style=none] (20) at (2, 2.75) {$Y$};
		\node [style=none] (21) at (4, 2.75) {$Y$};
		\node [style=none] (22) at (-3, -2.5) {$X$};
		\node [style=none] (23) at (3, -2.5) {$X$};
	\end{pgfonlayer}
	\begin{pgfonlayer}{edgelayer}
		\draw [style=none] (0.center) to (1);
		\draw [style=none, bend left=45] (1) to (2.center);
		\draw [style=none, bend right=45] (1) to (3.center);
		\draw [style=none] (7.center) to (8);
		\draw [style=none, bend left=45] (8) to (9.center);
		\draw [style=none, bend right=45] (8) to (10.center);
		\draw (12.center) to (5);
		\draw (13.center) to (6);
		\draw (14.center) to (9.center);
		\draw (15.center) to (11.center);
		\draw (2.center) to (5);
		\draw (4.center) to (6);
	\end{pgfonlayer}
\end{tikzpicture}	
	\]
\end{definition}

As per~\cite[Remark~10.13]{markov_cats}, the deterministic morphisms form a symmetric monoidal subcategory $\cC_\detc \subseteq \cC$ which contains all structure morphisms of $\cC$, including the comultiplications themselves. The fact that all morphisms are comonoid homomorphisms implies that $\cC_\detc$ is actually cartesian monoidal.

In $\FinStoch$, the deterministic morphisms $X \to Y$ are exactly the $\{0,1\}$-valued stochastic matrices, which are indeed those which are deterministic in the sense that they do not involve any randomness. Since these stochastic matrices are in obvious bijection with honest functions $X \to Y$, we conclude that $\FinStoch_\detc$ is equivalent to $\FinSet$. In $\Stoch$, the deterministic morphisms $f : X \to Y$ are exactly those Markov kernels for which $f(S|x) \in \{0,1\}$ for every $S \in \Sigma_Y$ and $x \in X$. In other words, as the term suggests, the deterministic morphisms are again those which do not involve any randomness; in $\BorelStoch$, the deterministic morphisms can be identified with the measurable functions. In $\CRing_+\op$, the deterministic morphisms $R \to S$ are precisely those morphisms which are represented by additive unital maps $S \to R$ which are in addition multiplicative, i.e.~the ring homomorphisms. Thus the subcategory of deterministic morphisms is exactly the opposite of the usual category of commutative rings.

We recall some more definitions that we will be using later on.

\begin{definition}[{\cite[Definition~11.31]{markov_cats}}]
	A Markov category $\cC$ is \emph{causal} if whenever
	
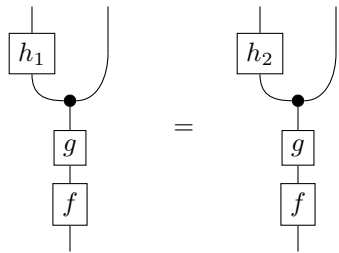
\[
		\begin{tikzpicture}
	\begin{pgfonlayer}{nodelayer}
		\node [style=bn] (0) at (-3, 0.75) {};
		\node [style=morphism] (1) at (-3, -2) {$f$};
		\node [style=none] (2) at (-3, -3.25) {};
		\node [style=morphism] (3) at (-4, 2) {$h_1$};
		\node [style=none] (4) at (-4, 3.25) {};
		\node [style=none] (5) at (-2, 2) {};
		\node [style=none] (6) at (-2, 3.25) {};
		\node [style=none] (7) at (0, 0) {$=$};
		\node [style=morphism] (8) at (-3, -0.5) {$g$};
		\node [style=bn] (9) at (3, 0.75) {};
		\node [style=morphism] (10) at (3, -2) {$f$};
		\node [style=none] (11) at (3, -3.25) {};
		\node [style=morphism] (12) at (2, 2) {$h_2$};
		\node [style=none] (13) at (2, 3.25) {};
		\node [style=none] (14) at (4, 2) {};
		\node [style=none] (15) at (4, 3.25) {};
		\node [style=morphism] (16) at (3, -0.5) {$g$};
	\end{pgfonlayer}
	\begin{pgfonlayer}{edgelayer}
		\draw (2.center) to (1);
		\draw (4.center) to (3);
		\draw [in=180, out=-90, looseness=1.25] (3) to (0);
		\draw (6.center) to (5.center);
		\draw [in=0, out=-90] (5.center) to (0);
		\draw (8) to (0);
		\draw (8) to (1);
		\draw (11.center) to (10);
		\draw (13.center) to (12);
		\draw [in=180, out=-90, looseness=1.25] (12) to (9);
		\draw (15.center) to (14.center);
		\draw [in=0, out=-90] (14.center) to (9);
		\draw (16) to (9);
		\draw (16) to (10);
	\end{pgfonlayer}
\end{tikzpicture}
	\]
	holds for morphisms as indicated, then also
	
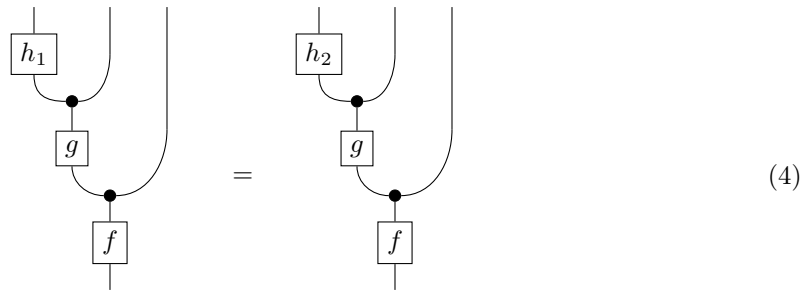
\begin{equation}
		\label{causal2}
		\begin{tikzpicture}
	\begin{pgfonlayer}{nodelayer}
		\node [style=bn] (0) at (-4.5, 2) {};
		\node [style=morphism] (1) at (-3.5, -1.75) {$f$};
		\node [style=none] (2) at (-3.5, -3) {};
		\node [style=morphism] (3) at (-5.5, 3.25) {$h_1$};
		\node [style=none] (4) at (-5.5, 4.5) {};
		\node [style=none] (5) at (-3.5, 3.25) {};
		\node [style=none] (6) at (-3.5, 4.5) {};
		\node [style=none] (7) at (0, 0) {$=$};
		\node [style=morphism] (8) at (-4.5, 0.75) {$g$};
		\node [style=bn] (9) at (-3.5, -0.5) {};
		\node [style=none] (10) at (-2, 1.25) {};
		\node [style=none] (11) at (-2, 4.5) {};
		\node [style=bn] (12) at (3, 2) {};
		\node [style=morphism] (13) at (4, -1.75) {$f$};
		\node [style=none] (14) at (4, -3) {};
		\node [style=morphism] (15) at (2, 3.25) {$h_2$};
		\node [style=none] (16) at (2, 4.5) {};
		\node [style=none] (17) at (4, 3.25) {};
		\node [style=none] (18) at (4, 4.5) {};
		\node [style=morphism] (19) at (3, 0.75) {$g$};
		\node [style=bn] (20) at (4, -0.5) {};
		\node [style=none] (21) at (5.5, 1.25) {};
		\node [style=none] (22) at (5.5, 4.5) {};
	\end{pgfonlayer}
	\begin{pgfonlayer}{edgelayer}
		\draw (2.center) to (1);
		\draw (4.center) to (3);
		\draw [in=180, out=-90, looseness=1.25] (3) to (0);
		\draw (6.center) to (5.center);
		\draw [in=0, out=-90] (5.center) to (0);
		\draw (8) to (0);
		\draw (9) to (1);
		\draw [in=-90, out=180] (9) to (8);
		\draw (11.center) to (10.center);
		\draw [in=0, out=-90] (10.center) to (9);
		\draw (14.center) to (13);
		\draw (16.center) to (15);
		\draw [in=180, out=-90, looseness=1.25] (15) to (12);
		\draw (18.center) to (17.center);
		\draw [in=0, out=-90] (17.center) to (12);
		\draw (19) to (12);
		\draw (20) to (13);
		\draw [in=-90, out=180] (20) to (19);
		\draw (22.center) to (21.center);
		\draw [in=0, out=-90] (21.center) to (20);
	\end{pgfonlayer}
\end{tikzpicture}
	\end{equation}
	\label{causal_defn}
\end{definition}

Intuitively, this condition states that if the choice between $h_1$ and $h_2$ is irrelevant for what happens ``in the future of $g$'', then that choice is likewise irrelevant for what happened ``in the past of $g$''.
It is known that $\Stoch$ is causal~\cite[Example~11.35]{markov_cats}, and
therefore so are the subcategories $\BorelStoch$ and $\FinStoch$.

\begin{example}
\label{cring_not_causal}
$\CRing_{+}^{\op}$ is \emph{not} causal.
To see this, consider the ring $\mathbb{Z}[t]$, where the following counterexample similar to~\cite[Example~11.33]{markov_cats} can be formulated.
Consider the additive unital maps given as the additive extensions of
\begin{align*}
	f(t^{n}) = {}& \begin{cases}t^{n-1} & n \geq 1 \\ 1 & n = 0 \end{cases}\\[3pt]
	g(t^{n}) = {}& \begin{cases}t & n \geq 1 \\ 1 & n= 0\end{cases}\\[3pt]
	h_{1}(t^{n}) = {}& t^{n}\\[3pt]
	h_{2}(t^{n}) = {}& 1
\end{align*}
On the basis monomials $t^n$, we have $(fg)(t^n) = 1$, and therefore
\[
	(fg)(h_1(t^n)t^m) = 1 = (fg)(h_2(t^n)t^m)
\]
for all $n,m \in \mathbb{N}$, so that the hypothesis holds.
If $\CRing_{+}^{\op}$ were causal, then we would also have
\[
	f(g(h_{1}(t^n) t^m) t^\ell) = f(g(h_{2}(t^n) t^m) t^\ell)
\]
for all $n,m,\ell \in \mathbb{N}$. But this clearly fails for $n = \ell = 1$ and $m = 0$.
\end{example}

The following definition goes back to Cho and Jacobs~\cite[Definition~5.1]{cho_jacobs}. A more detailed investigation of its properties can be found in~\cite[Section~13]{markov_cats}.

\begin{definition}
	\label{defnasequal}
	Given morphisms $p : A \to X$ and $f,g : X \to Y$ in a Markov category, we
	say that $f$ and $g$ are \emph{$p$-a.s.~equal} (or $p$-\emph{almost surely equal}) if
	
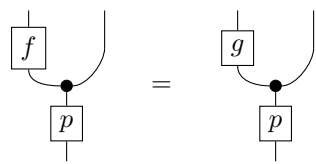
\[
		\begin{tikzpicture}
	\begin{pgfonlayer}{nodelayer}
		\node [style=morphism] (0) at (-2.5, -1) {$p$};
		\node [style=morphism] (1) at (-3.5, 1) {$f$};
		\node [style=none] (6) at (-2.5, -2) {};
		\node [style=none] (7) at (-3.5, 2) {};
		\node [style=none] (8) at (-1.5, 1) {};
		\node [style=none] (9) at (-1.5, 2) {};
		\node [style=bn] (10) at (-2.5, 0) {};
		\node [style=none] (22) at (0, 0) {$=$};
		\node [style=morphism] (23) at (3, -1) {$p$};
		\node [style=morphism] (24) at (2, 1) {$g$};
		\node [style=none] (25) at (3, -2) {};
		\node [style=none] (26) at (2, 2) {};
		\node [style=none] (27) at (4, 1) {};
		\node [style=none] (28) at (4, 2) {};
		\node [style=bn] (29) at (3, 0) {};
	\end{pgfonlayer}
	\begin{pgfonlayer}{edgelayer}
		\draw [in=-90, out=0, looseness=0.75] (10) to (8.center);
		\draw (8.center) to (9.center);
		\draw (10) to (0);
		\draw (0) to (6.center);
		\draw (1) to (7.center);
		\draw [in=180, out=-90] (1) to (10);
		\draw [in=-90, out=0, looseness=0.75] (29) to (27.center);
		\draw (27.center) to (28.center);
		\draw (29) to (23);
		\draw (23) to (25.center);
		\draw (24) to (26.center);
		\draw [in=180, out=-90] (24) to (29);
	\end{pgfonlayer}
\end{tikzpicture}
	\]
	and we also write this more concisely as $f =_{p\as} g$.
  \end{definition}

  This term comes from measure theory, and should be interpreted as meaning
  that for all those values of $X$ which can occur as outputs of $p$, feeding these values to $f$ and $g$ as inputs is guaranteed to result in the same distribution; at the same time, the behavior of $f$ and $g$ may be different on those input values which do not occur as outputs of $p$.

  We finally introduce a notion of conditional independence, which is a straightforward generalization of~\cite[Defnition~12.12]{markov_cats} from the binary case to the $n$-ary case.

\begin{definition}
	Given a morphism $p : A \to X_1 \otimes \ldots \otimes X_n$ in a Markov category, we say that $p$ \emph{displays the conditional independence} $X_1 \perp \ldots \perp X_n \mid\mid A$ if the equation
	\[
		\begin{tikzpicture}
	\begin{pgfonlayer}{nodelayer}
		\node [style=large morphism] (0) at (-3.25, 0) {$p$};
		\node [style=none] (1) at (-3.25, -0.5) {};
		\node [style=none] (2) at (-4.25, 0.5) {};
		\node [style=none] (3) at (-2.25, 0.5) {};
		\node [style=none] (4) at (-3.25, 1.25) {$\cdots$};
		\node [style=none] (5) at (-4.25, 1.75) {};
		\node [style=none] (6) at (-2.25, 1.75) {};
		\node [style=none] (7) at (-3.25, -1.5) {};
		\node [style=none] (8) at (-4.25, 2.25) {$X_1$};
		\node [style=none] (9) at (-2.25, 2.25) {$X_n$};
		\node [style=none] (10) at (-3.25, -2) {$A$};
		\node [style=none] (11) at (0, 0) {$=$};
		\node [style=large morphism] (12) at (3.25, 0.75) {$p$};
		\node [style=none] (13) at (3.25, 0.25) {};
		\node [style=none] (14) at (2.25, 1.25) {};
		\node [style=none] (15) at (4.25, 1.25) {};
		\node [style=none] (16) at (3.25, 2) {$\cdots$};
		\node [style=none] (17) at (2.25, 2.5) {};
		\node [style=bn] (18) at (4.25, 2.25) {};
		\node [style=none] (20) at (2.25, 3) {$X_1$};
		\node [style=none] (22) at (6, -3) {$A$};
		\node [style=none] (23) at (6, -2.5) {};
		\node [style=bn] (24) at (6, -1.5) {};
		\node [style=none] (25) at (6, 0.75) {$\cdots$};
		\node [style=large morphism] (26) at (8.75, 0.75) {$p$};
		\node [style=none] (27) at (8.75, 0.25) {};
		\node [style=none] (28) at (7.75, 1.25) {};
		\node [style=none] (29) at (9.75, 1.25) {};
		\node [style=none] (30) at (8.75, 2) {$\cdots$};
		\node [style=bn] (31) at (7.75, 2.25) {};
		\node [style=none] (32) at (9.75, 2.5) {};
		\node [style=none] (34) at (9.75, 3) {$X_n$};
		\node [style=none] (35) at (6, -0.75) {$\cdots$};
	\end{pgfonlayer}
	\begin{pgfonlayer}{edgelayer}
		\draw (5.center) to (2.center);
		\draw (6.center) to (3.center);
		\draw (1.center) to (7.center);
		\draw (17.center) to (14.center);
		\draw (18) to (15.center);
		\draw (24) to (23.center);
		\draw [in=-180, out=-90, looseness=0.75] (13.center) to (24);
		\draw (31) to (28.center);
		\draw (32.center) to (29.center);
		\draw [in=-90, out=0, looseness=0.75] (24) to (27.center);
	\end{pgfonlayer}
\end{tikzpicture}
	\]
	holds.
\end{definition}

Conditional independence does not depend on the order of the tensor factors: if $p$ displays the conditional independence $X_1 \perp \ldots \perp X_n \mid\mid A$, then composing $p$ with any permutation $\sigma$ of the tensor factors $X_1,\ldots,X_n$ gives a morphism $A \to X_{\sigma(1)} \otimes \ldots \otimes X_{\sigma(n)}$ which displays the conditional independence $X_{\sigma(1)} \perp \ldots \perp X_{\sigma(n)} \mid\mid A$. Therefore if $\bigotimes_{j \in F} X_j$ is a finite product without any particular order on the factors\footnote{See e.g.~\cite[Proposition~II.1.5]{DMOS} for how to make sense of tensor products in symmetric monoidal categories without a prescribed order of the factors.}, then for $p : A \to \bigotimes_{j \in F} X_j$ there is no ambiguity about whether $p$ displays conditional independence or not, and we also write $\perp_{j \in F} X_j \mid\mid A$ in case that it does.

\section{Infinite tensor products in semicartesian symmetric monoidal categories}
\label{infprod_semicartesian}

Many theorems in probability theory involve infinitely many random variables at a time, together with a joint distribution for them. In the setting of Markov categories, this means that we need to consider morphisms whose codomain is an ``infinite tensor product'' in a suitable sense. In this section, we start by considering infinite tensor products in the more general context of semicartesian symmetric monoidal categories.

Recall that the Kolmogorov's extension theorem states that joint distributions of a given (infinite) family of random variables are in bijection with compatible families of distributions of finite subsets of these variables. Here, a joint distribution of all of them determines such a compatible family by considering all its marginals on finite subsets. This motivates our general definition of infinite tensor products, which we now introduce. Suppose that $(X_i)_{i \in J}$ is any family of objects, where the indexing set $J$ is typically infinite. For any finite subset $F \subseteq J$, we also write $X_F := \bigotimes_{i \in F} X_i$ for simplicity of notation. If $F \subseteq F' \subseteq J$ are two finite subsets, then the fact that $\cC$ is semicartesian monoidal gives us \emph{marginalization morphisms}
\[
	\pi_{F',F} \: : \: X_{F'} \longrightarrow X_F.
\]
Via these maps, the finite tensor products $(X_F)_{F \subseteq J}$ make up a cofiltered diagram in the form of a functor from the poset of finite subsets of $J$, ordered by reverse inclusion, to $\cC$.

\begin{definition}
	\label{semicartesian_infproduct}
	Let $(X_i)_{i \in J}$ be a family of objects in $\cC$. An \emph{infinite tensor product}
	\[
		X_J \: := \: \bigotimes_{i \in J} X_i
	\]
	is the limit of the diagram $F \mapsto X_F$, indexed by the poset of finite subsets $F \subseteq J$ ordered by reverse inclusion, if this limit exists and is preserved by the functor $- \otimes Y$ for every object $Y \in \cC$.
\end{definition}

We will refer to the structure maps $\pi_F : X_J \to X_F$ as \emph{finite marginalizations}. The extra preservation means more concretely that $X_J \otimes Y$ is likewise the cofiltered limit of the $X_F \otimes Y$ with respect to the maps
\[
	\pi_F \otimes \id_Y \: : \: X_J \otimes Y \longrightarrow X_F \otimes Y.
\]
We will motivate this preservation condition in \Cref{preserve}.

Of course, an infinite tensor product may or may not exist in $\cC$. \Cref{infprods_finstoch} will present a simple example in which infinite tensor products only exist for \emph{finite} $J$.

\begin{remark}
	If $J$ is finite, then the infinite tensor product always exists and coincides with the tensor product $X_J = \bigotimes_{i \in J} X_i$ specified by the monoidal structure: then the finite product itself is initial in the defining diagram of the limit, and the preservation condition holds automatically for the same reason. Thus infinite tensor products $\bigotimes_{i \in J} X_i$ are of interest only for infinite $J$, as the terminology suggests. And for finite $J$, the notation is such that $\bigotimes_{i \in J} X_i$ can be interpreted either as specified directly by the monoidal structure, or as an infinite tensor product in the sense of \Cref{semicartesian_infproduct}. Since the result is the same (up to unique isomorphism), there is no ambiguity and the notation is consistent.
\end{remark}

\begin{remark}
	\label{preserve}
	It is natural to require different infinite tensor products to interact well with another. If $X_J$ is an infinite tensor product of a family $(X_i)_{i \in J}$ and we are given an additional object $X_\ast$ for $\ast \not \in J$, then there is a canonical comparison isomorphism
	\begin{equation}
		\label{comparison}
		X_J \otimes X_\ast \longrightarrow X_{J \,\sqcup\, \{\ast\}}
	\end{equation}
	induced via the universal property of $X_{J \,\sqcup\, \{\ast\}}$, which is the infinite tensor product of the original family with $X_\ast$ thrown in. In order for infinite tensor products to be coherent, one will want this comparison isomorphism to be an isomorphism. We now show that this is indeed the case thanks to the preservation condition in \Cref{semicartesian_infproduct}, even if $X_{J \,\sqcup\, \{\ast\}}$ does not exist a priori.
	
	More precisely, we exhibit $X_J \otimes X_\ast$ as the defining cofiltered limit of the infinite tensor product $X_{J \,\sqcup\, \{\ast\}}$ as follows. The object $X_J \otimes X_\ast$ can be equipped with finite marginalization morphisms with respect to $F \subseteq J \sqcup \{\ast\}$ given by
	\[
		\pi_{F \setminus \{\ast\}} \otimes \id \: : \: X_J \otimes X_\ast \longrightarrow X_{F \setminus \{\ast\}} \otimes X_\ast
	\]
	whenever $\ast \in F$, and just by
	\[
		\pi_F \otimes \discard{X_\ast} : X_J \otimes X_\ast \longrightarrow X_F 
	\]
	in case that $\ast \not \in F$. It is straightforward to check that these morphisms indeed make $X_J \otimes X_\ast$ carry the universal property of the infinite tensor product $X_{J \sqcup \{\ast\}}$ if and only if the functor $- \otimes X_\ast$ preserves the defining cofiltered limit of the infinite tensor product $X_J$. It is also easy to see that this new cofiltered limit is again preserved by every functor $- \otimes Y$, based on the assumption that the defining cofiltered limit of the original infinite tensor product is preserved by $- \otimes (X_\ast \otimes Y)$ as well.

	In summary, the preservation condition in this definition amounts to the requirement that the canonical comparison morphism~\eqref{comparison} must be an isomorphism, as one would expect intuitively from a good notion of infinite tensor product. 
\end{remark}

\begin{example}
	\label{infprods_cring}
	The dual definition of infinite tensor products of algebraic structures as \emph{filtered colimits} of finite tensor products is well-known in the literature, e.g.~in the case of C*-algebras~\cite[p.~315]{blackadar}, although we do not know whether the preservation condition has been made explicit before. For example in the case where our category is $\CRing_+\op$, we recover the usual folklore definition of an infinite tensor product of rings $\bigotimes_{i \in J} R_i$ in terms of formal sums of elementary tensors, where an elementary tensor is a family of elements $(r_i)_{i \in J}$ such that all but finitely many are equal to the respective unit. The necessary preservation condition is easily seen to hold. Note that we do not yet consider the multiplication on $\bigotimes_{i \in J} R_i$, since \Cref{semicartesian_infproduct} is not yet concerned with the comonoid structures on the objects. We will get to this in the next section.
	
	Intuitively, our definition matches up with the known algebraic ones under the categorical duality of algebra and geometry, where our \Cref{semicartesian_infproduct} is on the geometrical side of the duality.
\end{example}

\begin{remark}
	In those semicartesian monoidal categories $\cC$ that are of interest to us, morphisms $I \to X$ play the role of \emph{probability measures} on $X$. Thus applying the defining universal property of an infinite tensor product with respect to maps out of $I$ implements the \emph{Kolmogorov extension theorem} in $\cC$: probability measures on an infinite product $X_J$ are in bijection with consistent families of probability measures on the finite products $X_F$ for $F \subseteq J$. The spirit of \Cref{semicartesian_infproduct} is thus such that it turns the Kolmogorov extension theorem into a definition.
\end{remark}

\begin{example}

	\label{infprods_stoch}
	We now consider infinite tensor products in $\Stoch$, where we would like infinite tensor products to be given by the corresponding infinite products of measurable spaces in the usual sense. Since the Kolmogorov extension theorem does not hold for general (even merely countable) products of measurable spaces~\cite{AJ}, this is not the case without further additional assumptions on the measurable spaces involved. Thus we do not know whether $\Stoch$ has infinite tensor products, although we suspect that it does not; but even in the unlikely case that it does, they are in general not the ones that one will care about for the purposes of probability theory.

	However, the situation is much better for countable products in $\BorelStoch$, for which the Kolmogorov extension theorem indeed holds~\cite[Theorem~14.35]{klenke}. In other words, if $\left( (X_i, \Sigma_i) \right)_{i \in \bN}$ is a sequence of standard Borel spaces, then the cartesian product $X_\bN = \prod_{i \in \bN} X_i$, when equipped with the product $\sigma$-algebra $\Sigma_\bN$, satisfies the universal property of an infinite tensor product with respect to maps out of $I$, since the marginalization maps implement a bijection between probability measures on $X_{\bN} = \prod_{i \in \bN} X_i$ and compatible families of probability measures on the finite subproducts $X_F = \prod_{i \in F} X_i$, where the compatibility is with respect to marginalization to smaller subproducts specified by $F' \subseteq F$.
	
	We now prove that this implies the universal property in general: if $A$ and all the $X_i$ are standard Borel spaces, then Markov kernels $A \to X_J$ are in bijection with compatible families of Markov kernels $A \to X_F$. While we just saw that this is the case for $A = I$, we now show that it holds for arbitrary $A$. Thus suppose that
	\[
		\Big( g_F : (A,\Sigma_A) \longrightarrow (X_F, \Sigma_F) \Big)_{F \subseteq J \text{ finite}}
	\]
	is a family of Markov kernels satisfying the compatibility condition $\pi_{F',F} \circ g_{F'} = g_F$ for all finite $F \subseteq F' \subseteq J$. Then for every $a \in A$, the probability measures $g_F(-|a) : \Sigma_F \to [0,1]$ are a compatible family to which the Kolmogorov extension theorem in the form of~\cite[Theorem~14.35]{klenke} applies, and we obtain a unique probability measure $g_J(-|a) : \Sigma_J \to [0,1]$ which has the original $g_F(-|a)$ as its finite marginals. It remains to be shown that for every $S \in \Sigma_J$, the map $a \mapsto g_J(S|a)$ is measurable. Since limits of pointwise convergent sequences of measurable real-valued functions are again measurable, the set of $S$ for which this measurability holds is closed under countable disjoint union, and it is clearly closed under complements. We therefore have a $\lambda$-system. Since the map is measurable by assumption whenever $S$ is a measurable cylinder set, and the cylinder sets form a $\pi$-system, the $\pi$-$\lambda$-theorem implies that $a \mapsto g_J(S|a)$ is measurable for all $S$ in the $\sigma$-algebra generated by the measurable cylinder sets, which is exactly the product $\sigma$-algebra $\Sigma_J$. Hence the existence part of the universal property indeed holds; and the uniqueness part is obvious by considering each $a \in A$ separately and applying the Kolmogorov extension theorem.

	Since the resulting comparison morphism \Cref{comparison} is an isomorphism by construction---both sides are given by the corresponding products of measurable spaces---it follows that the required preservation condition of infinite tensor products holds as well.

	In conclusion, $\BorelStoch$ indeed has countable tensor products. It is plausible that there is another subcategory of $\Stoch$ strictly larger than $\BorelStoch$ which has all infinite tensor products; for example, one can try to construct such a subcategory by imposing compactness or perfectness conditions~\cite[\S{451}]{fremlin4} on the Markov kernels, since under such assumption one again has a version of the Kolmogorov extension theorem~\cite[Corollary~454G]{fremlin4}. However, we have so far not been able to find a definition of ``compact Markov kernel'' or ``perfect Markov kernel'' which would guarantee closure under composition, due to problems with perfect measures not being stable under mixtures~\cite{ramachandran}. Hence the problem of finding a subcategory of $\Stoch$ with all infinite tensor products remains open.
\end{example}

\begin{example}
	\label{infprods_finstoch}
	Infinite tensor products never exist in $\FinStoch$, in the following sense: if the family $(X_i)_{i \in J}$ is such that no $X_i$ is empty and infinitely many of them contain at least two elements, then $\bigotimes_i X_i$ does not exist. One way to see this is to use the fact that the hom-sets of $\FinStoch$ are convex subsets of the vector space of matrices, and that composition distributes over these convex combinations, making $\FinStoch$ into a category enriched in convex sets.\footnote{See also \href{https://golem.ph.utexas.edu/category/2010/09/what_is_this_category_enriched.html}{golem.ph.utexas.edu/category/2010/09/what\_is\_this\_category\_enriched.html} for discussion of this enrichment.} Since these hom-sets are finite-dimensional, for every $Y$ and $Z$ there is $n \in \bN$ such that among every $n$ morphisms $Y \to Z$, one of them can be written as a convex combination of the others. Now suppose that the product $\bigotimes_i X_i$ existed. Then by choosing varying elements of each $X_j$, we could construct uncountably many morphisms $1 \to \bigotimes_i X_i$ whose marginalizations $1 \to X_F$ are all deterministic. Per the above, one of these hypothetical morphisms can be written as a convex combination of finitely many others. By choosing $F \subseteq J$ suitably, we can achieve that the finite marginalizations $1 \to X_F$ of the finitely many morphisms involved in this convex combination are all distinct. Since these morphisms $1 \to X_F$ are deterministic by construction, and no deterministic morphism in $\FinStoch$ can be written as a convex combination of other deterministic morphisms, we have arrived at a contradiction. Thus $\bigotimes_i X_i$ does not exist.
\end{example}

\begin{example}[Paolo Perrone, personal communication]
	\label{poset_case}
	Suppose that $\cC$ is a partially ordered set, considered as a category with $X \geq Y$ if and only if the morphism $X \to Y$ exists.\footnote{This direction of the ordering is opposite to the usual convention, but necessary for the following to work without reversing direction.} Then the monoidal structure makes $\cC$ into a preordered commutative monoid, which we write additively. Here, the functoriality of the monoidal structure amounts to the assumption that addition is monotone,
	\[
		X \geq Y \quad\: \Longrightarrow \quad\: X + Z \geq Y + Z.
	\]
	Furthermore, the monoidal structure is semicartesian if and only if $X \geq 0$ for every $X$. One can then apply \Cref{infprod_operadic} to show that $\cC$ has infinite tensor products if and only if $\cC$ is a \emph{finitary complete monoid} in the sense of Goldstern and Karner~\cite[Section~2.2]{karner}, and these infinite tensor products then correspond to infinite sums $\sum_{i \in J} X_i$ in the monoid.
\end{example}

We return to the general theory.

\begin{lemma}
	\label{two_infproducts}
	Let $J = J_1 \,\sqcup\, J_2$ be a disjoint union and $(X_i)_{i \in J}$ a family of objects in $\cC$. Suppose that the infinite tensor products $X_{J_1}$ and $X_{J_2}$ exist. Then the object
	\[
		X_{J_1} \otimes X_{J_2} 
	\]
	is an infinite tensor product $X_J$ with respect to the finite marginalizations morphisms given by, for every finite $F \subseteq J$,
	\[
		\begin{tikzcd}[column sep=huge]
			\rho_F \: : \: X_{J_1} \otimes X_{J_2} \ar{r}{\pi_{F\cap J_1} \,\otimes\, \pi_{F\cap J_2}} & X_{F \cap J_1} \otimes X_{F \cap J_2}.
		\end{tikzcd}
	\]
\end{lemma}
\begin{proof}
	For fixed finite $F_1 \subseteq J_1$, the morphisms
	\[
		\id \otimes \pi_{F_2} \: : \: X_{F_1} \otimes X_{J_2} \longrightarrow X_{F_1} \otimes X_{F_2}
	\]
	exhibit $X_{F_1} \otimes X_{J_2}$ as the cofiltered limit of the $X_{F_1} \otimes X_{F_2}$ by the preservation assumption. Similarly for varying $F_1$, the morphisms
	\[
		\pi_{F_1} \otimes \id \: : \: X_{J_1} \otimes X_{J_2} \longrightarrow X_{F_1} \otimes X_{J_2}
	\]
	exhibit $X_{J_1} \otimes X_{J_2}$ as the cofiltered limit of the $X_{F_1} \otimes X_{J_2}$. The claimed universal property follows since a limit of limits is a limit, and it is easy to see that the diagram shapes match up: the poset of finite subsets of $J_1 \sqcup J_2$ is the product of the posets of finite subsets of $J_1$ and $J_2$. The preservation property with respect to applying a tensoring functor $- \otimes Y$ carries along both of the above limits individually.
\end{proof}

More generally, we may consider infinite tensor products of infinite tensor products, $\bigotimes_{k \in K} \bigotimes_{j \in J_k} X_{k,j}$, where now $(J_k)_{k \in K}$ is a family of sets and $(X_{k,j})$ a doubly indexed family of objects. One may think that this doubly infinite tensor product should be isomorphic to the single-step infinite tensor product $\bigotimes_{k \in K, \: j \in J_k} X_{k,j}$. And indeed, if for every finite $F \subseteq \prod_{k \in K} J_k$ we choose any finite $G \subseteq K$ such that $(k,j) \in F$ implies $k \in G$, then we have morphisms
\begin{equation}
	\label{double_infproduct}
	\begin{tikzcd}[column sep=1cm]
		\rho_F \; : \; \bigotimes_{k \in K} \bigotimes_{j \in J_k} X_{k,j} \ar{r}{\pi_G} & \bigotimes_{k \in G} \bigotimes_{j \in J_k} X_{k,j} \ar{rr}{\bigotimes_{k \in G} \pi_{F \cap J_k}} & & \bigotimes_{k \in G, \: j \in F \cap J_k} X_{k,j}
	\end{tikzcd}
\end{equation}
where the object on the right can also be written as $\bigotimes_{(k,j) \in F} X_{k,j}$. It is straightforward to see that this $\rho_F$ does not depend on the particular choice of $G$ (up to monoidal structure isomorphisms). We now show that these $\rho_F$'s are finite marginalization morphisms which make the doubly infinite tensor product $\bigotimes_{k \in K} \bigotimes_{j \in J_k} X_{k,j}$ into the single infinite tensor product $\bigotimes_{k \in K, \: j \in J_k} X_{k,j}$.

\begin{lemma}
	\label{infprod_operadic}
	Suppose that $(J_k)_{k \in K}$ is a family of sets and $(X_{k,j})_{j \in J_k, k \in K}$ a doubly indexed family of objects. If all of the infinite tensor products $\bigotimes_{j \in J_k} X_{j,k}$ and $\bigotimes_{k \in K} \bigotimes_{j \in J_k} X_{k,j}$ exist, then the morphisms~\eqref{double_infproduct} exhibit $\bigotimes_{k \in K} \bigotimes_{j \in J_k} X_{k,j}$ as the infinite tensor product $\bigotimes_{k \in K, \: j \in J_k} X_{k,j}$.
\end{lemma}
\begin{proof}
	By repeated application of \Cref{two_infproducts}, the second half of~\Cref{double_infproduct} exhibits $\bigotimes_{k \in G} \bigotimes_{j \in J_k} X_{k,j}$ as the infinite tensor product of the $X_{k,j}$ for $k \in G$ and $j \in J_k$. The relevant universal property as well as its preservation by $- \otimes Y$ then follow from another straightforward diagram chase.
\end{proof}

\section{Infinite tensor products in Markov categories}
\label{infprod_markov}

If a Markov category $\cC$ has infinite tensor products, then one should expect a compatibility condition between these infinite tensor products and the comonoid structures on the objects. This is imposed as follows. % In the following string diagrams, we draw the object associated to an infinite tensor product as a double wire.

\begin{definition}
    \label{defn_kolmogorov_ext}
    Let $\cC$ be a Markov category and $(X_i)_{i \in J}$ a family of objects. We say that an infinite tensor product $X_J = \bigotimes_{i \in J} X_i$ is a \emph{Kolmogorov product} if the finite marginalization morphisms $\pi_F : X_J \to X_F$ are deterministic.
\end{definition}

In $\BorelStoch$, the infinite tensor products constructed in \Cref{infprods_stoch} are Kolmogorov products. In $\CRing_+\op$, the infinite tensor products from \Cref{infprods_cring} are Kolmogorov products as well, provided that one equips $\bigotimes_i R_i$ with the tensor product ring structure, since this one is the only one which makes the canonical inclusions $R_j \to \bigotimes_i R_i$ into ring homomorphisms. These examples illustrate the following important point. In some Markov categories, such as $\CRing_+\op$, not all isomorphisms are deterministic, and there may even be isomorphic objects which have no deterministic isomorphism~\cite[Remark~10.10]{markov_cats}. In these categories, whether a given infinite tensor product is a Kolmogorov product or not generally depends on the specific choice of that object. In particular, there may be infinite tensor products which are not Kolmogorov products. For example in $\CRing_+\op$, we may take the tensor product of abelian groups $\bigotimes_i R_i$ and equip it with a different multiplication than the tensor product of rings one (as long as it has the same unit), and one will then have an infinite tensor product in $\CRing_+\op$ which is not a Kolmogorov product.

\begin{remark}
	If the Markov category $\cC$ is positive in the sense of~\cite[Definition~11.22]{markov_cats}, then a morphism $A \to B \otimes C$ is deterministic if and only if both of its marginals $A \to B$ and $A \to C$ are deterministic~\cite[Corollary~12.15]{markov_cats}. In this case, it follows that the finite marginalizations $\pi_F$ are deterministic if and only if all the single-factor marginalizations $\pi_{\{i\}} : X_J \to X_i$ are. Thus for an infinite tensor product to be a Kolmogorov product, it is then enough to verify that the $\pi_{\{i\}}$ are deterministic.
\end{remark}

Using the multiplicativity of the comonoid structure
of~\eqref{multiplicativity}, we see that \Cref{two_infproducts,infprod_operadic}
hold also for Kolmogorov products, in the sense that tensoring together
Kolmogorov products, or taking Kolmogorov products of Kolmogorov products, results
again in a Kolmogorov product.

\begin{proposition}
	Every Kolmogorov product is a categorical product in the cartesian monoidal subcategory of deterministic morphisms $\cC_\detc$.
	\label{kolmprod_catdet}
\end{proposition}
\begin{proof}
	For $(X_i)_{i \in J}$ a family of objects with Kolmogorov product $X_J$ and finite marginalizations $\pi_F : X_J \to X_F$, we first show that if we have a compatible family of deterministic morphisms $g_F : A \to X_F$, then also the induced $g_J : A \to X_J$ is deterministic. Drawing the Kolmogorov product $X_J$ as a double wire, we need to prove the determinism equation
	\[
		\begin{tikzpicture}
	\begin{pgfonlayer}{nodelayer}
		\node [style=none] (0) at (-3, -2) {};
		\node [style=bn] (1) at (-3, -0.5) {};
		\node [style=none] (2) at (-4, 0.5) {};
		\node [style=none] (3) at (-2, 0.5) {};
		\node [style=none] (4) at (-2, 0.5) {};
		\node [style=morphism] (5) at (-4, 0.75) {$g_J$};
		\node [style=morphism] (6) at (-2, 0.75) {$g_J$};
		\node [style=none] (7) at (3, -2) {};
		\node [style=bn] (8) at (3, 0) {};
		\node [style=none] (9) at (2, 1) {};
		\node [style=none] (10) at (4, 1) {};
		\node [style=none] (11) at (4, 1) {};
		\node [style=none] (12) at (-4, 2) {};
		\node [style=none] (13) at (-2, 2) {};
		\node [style=none] (14) at (2, 2) {};
		\node [style=none] (15) at (4, 2) {};
		\node [style=none] (16) at (0, 0) {$=$};
		\node [style=morphism] (17) at (3, -1) {$g_J$};
	\end{pgfonlayer}
	\begin{pgfonlayer}{edgelayer}
		\draw [style=none] (0.center) to (1);
		\draw [style=none, bend left=45] (1) to (2.center);
		\draw [style=none, bend right=45] (1) to (3.center);
		\draw [style=double line, bend left=45] (8) to (9.center);
		\draw [style=double line, bend right=45] (8) to (10.center);
		\draw [style=double line] (12.center) to (5);
		\draw [style=double line, in=90, out=-90] (13.center) to (6);
		\draw [style=double line] (14.center) to (9.center);
		\draw [style=double line] (15.center) to (11.center);
		\draw (2.center) to (5);
		\draw (4.center) to (6);
		\draw [style=double line] (8) to (17);
		\draw (17) to (7.center);
	\end{pgfonlayer}
\end{tikzpicture}
	\]
	By \Cref{two_infproducts}, the codomain object $X_J \otimes X_J$ is itself an infinite tensor product, and it is therefore enough to prove that for every finite $F \subseteq J$,
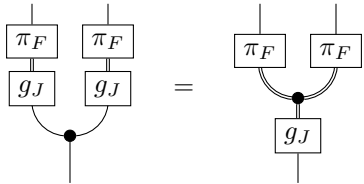
	\[
		\begin{tikzpicture}
	\begin{pgfonlayer}{nodelayer}
		\node [style=none] (0) at (-3, -2.5) {};
		\node [style=bn] (1) at (-3, -1.25) {};
		\node [style=none] (2) at (-4, -0.25) {};
		\node [style=none] (3) at (-2, -0.25) {};
		\node [style=none] (4) at (-2, -0.25) {};
		\node [style=morphism] (5) at (-4, 0) {$g_J$};
		\node [style=morphism] (6) at (-2, 0) {$g_J$};
		\node [style=none] (7) at (3, -2.5) {};
		\node [style=bn] (8) at (3, -0.25) {};
		\node [style=none] (9) at (2, 0.75) {};
		\node [style=none] (10) at (4, 0.75) {};
		\node [style=none] (11) at (4, 0.75) {};
		\node [style=morphism] (12) at (-4, 1.25) {$\pi_F$};
		\node [style=morphism] (13) at (-2, 1.25) {$\pi_F$};
		\node [style=morphism] (14) at (2, 1) {$\pi_F$};
		\node [style=morphism] (15) at (4, 1) {$\pi_F$};
		\node [style=none] (16) at (0, 0) {$=$};
		\node [style=morphism] (17) at (3, -1.25) {$g_J$};
		\node [style=none] (18) at (-4, 2.25) {};
		\node [style=none] (19) at (-2, 2.25) {};
		\node [style=none] (20) at (2, 2.25) {};
		\node [style=none] (21) at (4, 2.25) {};
	\end{pgfonlayer}
	\begin{pgfonlayer}{edgelayer}
		\draw [style=none] (0.center) to (1);
		\draw [style=none, bend left=45] (1) to (2.center);
		\draw [style=none, bend right=45] (1) to (3.center);
		\draw [style=double line, bend left=45] (8) to (9.center);
		\draw [style=double line, bend right=45] (8) to (10.center);
		\draw [style=double line] (12) to (5);
		\draw [style=double line, in=90, out=-90] (13) to (6);
		\draw [style=double line] (14) to (9.center);
		\draw [style=double line] (15) to (11.center);
		\draw (2.center) to (5);
		\draw (4.center) to (6);
		\draw [style=double line] (8) to (17);
		\draw (17) to (7.center);
		\draw (19.center) to (13);
		\draw (18.center) to (12);
		\draw (21.center) to (11.center);
		\draw (14) to (20.center);
	\end{pgfonlayer}
\end{tikzpicture}
	\]
	where we have assumed without loss of generality that the two finite marginalization maps are the same, which we can by replacing the corresponding finite subsets by their union. Using the determinism assumption for $\pi_F$ on the right, as well as the assumption that $g_F = \pi_F \circ g_J$ is deterministic, implies the claim.

	Now since every finite tensor product $X_F$ is a categorical product in $\cC_\detc$, we can use the fact that in every category,
    	\[
		\lim_{F \subseteq J \text{ finite}} \: \prod_{j \in F} X_j \: \cong \: \prod_{j \in J} X_j,
	\]
	and the claim follows.
\end{proof}

In particular, this proves that \Cref{defn_kolmogorov_ext} determines the
comonoid structure on a Kolmogorov product uniquely: $\cop{\bigotimes_i X_i}$
has to be the equal to the corresponding diagonal morphism in $\cC_\detc$, and
\Cref{kolmprod_catdet} determines $\bigotimes_i X_i$ up to unique \emph{deterministic}
isomorphism.

\begin{example}
    \label{setmulti}
    Let $\SetMulti$ be the category with objects ordinary sets, morphisms $f: X \to Y$ given by \emph{multivalued functions}, which are maps $f: X \to \cP(Y)\setminus \emptyset$, with composition of morphisms $f : X \to Y$ and $g : Y \to Z$ given by
    	\begin{equation}
		\label{relational_composition}
		(g\circ f)(x) := \bigcup_{y \in f(x)} g(y).
      \end{equation}
    $\SetMulti$ could also be described as the Kleisli category of the
    \emph{nonempty powerset monad}, which sends a set $X$ to $\cP(X) \setminus \emptyset$. In particular, $\FinSetMulti$ is closely related to $\Rel$, the category of sets and relations, since the latter can equivalently be regarded as the Kleisli category of the powerset monad (where the empty set is allowed). Intuitively, $\SetMulti$ is the closest relative of $\Rel$ to which our formalism can be applied. 
    See also \cite[Example~2.6]{markov_cats} for a treatment of the full subcategory $\FinSetMulti$ on finite sets.

    The Cartesian product of sets equips $\SetMulti$ with a symmetric monoidal structure, if for $f : A \to X$ and $g: B \to Y$ we set
	\[
	(f \tensor g)(a,b) := f(a) \times g(b) \:\subseteq\: X \times Y.
	\]
	The monoidal unit is any one-element set $I$. With the obvious structure isomorphisms inherited from $\Set$, this turns $\SetMulti$ into a semicartesian symmetric monoidal category, and we can give it a Markov category structure by setting
    $\cop{X}(x) := \{(x,x)\} \subset X \times X$.
    Then the deterministic morphisms are precisely those sending each point into a singleton set, resulting in an isomorphism of categories
    $\SetMulti_\detc \cong \Set$.

    Of course, $\Set$ has all infinite products. But perhaps surprisingly, these products are not Kolmogorov products, or even infinite tensor products in $\SetMulti$.
    To see this, consider for example the family $(X_n)_{n\in\bN}$ with $X_n = \{0,1\}$ for all $n$. 
    A map $I \to \prod_{n\in \bN} \{0,1\}$ is just a nonempty subset of this infinite product.
    Given such a subset $U$, the finite marginals $I \to \prod_{n \in F} \{0,1\}$ for finite $F \subset \bN$ are the images of $U$ under the projection map.
    But these images do not determine $U$. For instance, consider the subset $U$ of those sequences containing at least one $1$.
    Its image in every $\prod_{n \in F} \{0,1\}$ is the whole set, since any finite sequence of $0$s and $1$s can be extended to one containing a $1$, but $U$ itself is a proper subset. In particular, $U$ has the same compatible family of finite marginalizations as e.g.~the set of those sequences containing at least one $0$, and hence the uniqueness part of the universal property fails.
    
    It is clear that this counterexample is quite generic and works for any infinite family of nonempty sets, infinitely many of which contain at least two elements, as in \Cref{infprods_finstoch}. Thus $\SetMulti$ is very far from having Kolmogorov products. A good way to remedy this deficiency is to consider topological spaces instead of sets: as we will see in \Cref{examples}, the Kleisli category of the lower Vietoris monad on $\Top$ does have Kolmogorov products, and it can be thought of as the topological counterpart of $\SetMulti$.
\end{example}

The concept of Kolmogorov product describes a notion of infinite collection of random variables whose joint distribution is encoded in its finite marginals. It is natural to postulate the same for when such an infinite collection is considered to be independent.

\begin{definition}
	Let $X_J = \bigotimes_{i \in J} X_i$ be a Kolmogorov product in a Markov category. We say that a morphism $p : A \to X_J$ \emph{displays the conditional independence} $\perp_{i \in J}  X_i \mid\mid A$ if each finite marginalization $p_F : A \to X_F$ displays the conditional independence $\perp_{i \in F} X_i \mid\mid A$.
\end{definition}

\section{The zero--one laws}
\label{zeroones}

We now make use of the theory developed in the previous sections in order to state and prove synthetic versions of the classical zero--one laws of Kolmogorov and Hewitt--Savage.
Throughout this section, $(X_i)_{i \in J}$ is a family of objects in a Markov category $\cC$ with Kolmogorov product $X_J$ and finite marginalizations $\pi_F : X_J \to X_F$.

\begin{lemma}[The infinite independence lemma]
  \label{infindeplemma}
	Suppose $p: A \to X_J$ displays the conditional independence $\perp_i X_i \mid\mid A$.
	Then for every $i \in J$, it also displays the independence $X_i \perp X_{J \setminus \{i\}} \mid\mid A$.
\end{lemma}
\begin{proof}
	This is nontrivial because the assumption by definition only involves conditional independence of the finite marginalizations. What we need to prove is that
	
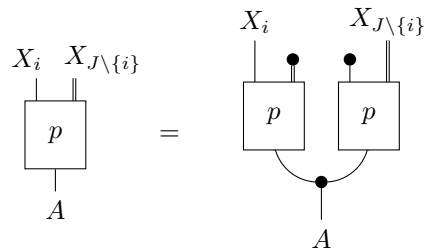
\[
		\begin{tikzpicture}
	\begin{pgfonlayer}{nodelayer}
		\node [style=medium box] (0) at (-3, 0) {$p$};
		\node [style=none] (1) at (-3, -1.5) {};
		\node [style=none] (2) at (-3.5, 0.5) {};
		\node [style=none] (3) at (-2.5, 0.5) {};
		\node [style=none] (4) at (-3.5, 1.5) {};
		\node [style=none] (5) at (-2.5, 1.5) {};
		\node [style=none] (6) at (-3, -2) {$A$};
		\node [style=none] (7) at (-3.75, 2) {$X_i$};
		\node [style=none] (8) at (-1.75, 2) {$X_{J\setminus\{i\}}$};
		\node [style=bn] (9) at (4, -1.25) {};
		\node [style=none] (10) at (4, -2.25) {};
		\node [style=none] (11) at (4, -2.75) {$A$};
		\node [style=none] (12) at (2.75, 0) {};
		\node [style=none] (13) at (5.25, 0) {};
		\node [style=medium box] (14) at (2.75, 0.5) {$p$};
		\node [style=medium box] (15) at (5.25, 0.5) {$p$};
		\node [style=none] (16) at (2.25, 1) {};
		\node [style=none] (17) at (3.25, 1) {};
		\node [style=none] (18) at (4.75, 1) {};
		\node [style=none] (19) at (5.75, 1) {};
		\node [style=bn] (20) at (4.75, 2) {};
		\node [style=bn] (21) at (3.25, 2) {};
		\node [style=none] (22) at (2.25, 2.5) {};
		\node [style=none] (23) at (5.75, 2.5) {};
		\node [style=none] (24) at (2.25, 3) {$X_i$};
		\node [style=none] (25) at (5.75, 3) {$X_{J\setminus\{i\}}$};
		\node [style=none] (26) at (0, 0) {$=$};
	\end{pgfonlayer}
	\begin{pgfonlayer}{edgelayer}
		\draw (4.center) to (2.center);
		\draw [style=double line] (5.center) to (3.center);
		\draw (0) to (1.center);
		\draw [bend right=45] (12.center) to (9);
		\draw (9) to (10.center);
		\draw [bend right=45] (9) to (13.center);
		\draw (16.center) to (22.center);
		\draw (20) to (18.center);
		\draw (15) to (13.center);
		\draw (14) to (12.center);
		\draw [style=double line] (17.center) to (21);
		\draw [style=double line] (19.center) to (23.center);
	\end{pgfonlayer}
\end{tikzpicture}
	\]
	As the codomain is exactly the Kolmogorov product $X_J$, it is enough to prove this equation upon composing the second output with the finite marginalizations $\pi_F : X_{J\setminus\{i\}} \to X_F$. But this is a consequence of the assumption.
\end{proof}

\begin{lemma}[The determinism lemma]
	\label{determinismlemma}
	Suppose that we have $p: A \to X$ and deterministic $s : X \to T$. If the joint
	
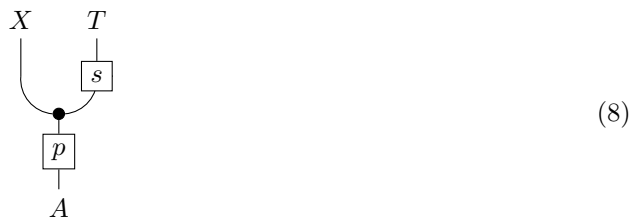
\begin{equation}
		\label{determinismlemma_joint}
		\begin{tikzpicture}
	\begin{pgfonlayer}{nodelayer}
		\node [style=bn] (0) at (0, 0) {};
		\node [style=none] (1) at (-1, 1) {};
		\node [style=none] (2) at (0, -2) {};
		\node [style=none] (3) at (1, 1) {};
		\node [style=morphism] (4) at (0, -1) {$p$};
		\node [style=morphism] (5) at (1, 1) {$s$};
		\node [style=none] (6) at (-1, 2) {};
		\node [style=none] (7) at (1, 2) {};
		\node [style=none] (8) at (-1, 2.5) {$X$};
		\node [style=none] (9) at (1, 2.5) {$T$};
		\node [style=none] (10) at (0, -2.5) {$A$};
	\end{pgfonlayer}
	\begin{pgfonlayer}{edgelayer}
		\draw (7.center) to (5);
		\draw (5) to (3.center);
		\draw (3.center) to (5);
		\draw [bend left=45] (3.center) to (0);
		\draw [bend left=45] (0) to (1.center);
		\draw (1.center) to (6.center);
		\draw (0) to (4);
		\draw (4) to (2.center);
	\end{pgfonlayer}
\end{tikzpicture}	
	\end{equation}
	displays the conditional independence $X \perp T \mid\mid A$, then the composite $sp: A \to T$ is deterministic.
\end{lemma}
This result can be thought of as a partial converse to the condition which characterizes positivity of a Markov category~\cite[Definition~11.22]{markov_cats}. Here and in the following, our notation ``$s$'' indicates that we think of this morphism as a statistic~\cite[Definition~14.2]{markov_cats}.
\begin{proof}
	Since the marginals of \Cref{determinismlemma_joint} are exactly $p$ and $sp$, the conditional independence assumption is
	
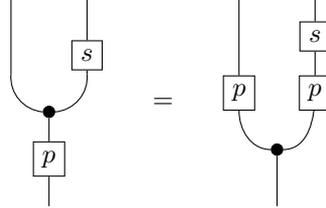
\[
		\begin{tikzpicture}
	\begin{pgfonlayer}{nodelayer}
		\node [style=bn] (0) at (-3, -0.25) {};
		\node [style=none] (1) at (-4, 0.75) {};
		\node [style=none] (2) at (-2, 0.75) {};
		\node [style=none] (3) at (-2, 0.75) {};
		\node [style=morphism] (4) at (-2, 1.25) {$s$};
		\node [style=bn] (5) at (3, -1.25) {};
		\node [style=none] (6) at (4, 0.25) {};
		\node [style=morphism] (7) at (4, 0.25) {$p$};
		\node [style=none] (8) at (-4, 2.75) {};
		\node [style=none] (9) at (-2, 2.75) {};
		\node [style=none] (10) at (2, 2.75) {};
		\node [style=morphism] (11) at (4, 1.75) {$s$};
		\node [style=none] (12) at (0, 0) {$=$};
		\node [style=morphism] (13) at (-3, -1.5) {$p$};
		\node [style=morphism] (14) at (2, 0.25) {$p$};
		\node [style=none] (15) at (4, 2.75) {};
		\node [style=none] (16) at (-3, -2.75) {};
		\node [style=none] (17) at (3, -2.75) {};
	\end{pgfonlayer}
	\begin{pgfonlayer}{edgelayer}
		\draw [style=none, bend left=45] (0) to (1.center);
		\draw [style=none, bend right=45] (0) to (2.center);
		\draw [style=none, in=-90, out=0] (5) to (6.center);
		\draw (9.center) to (4);
		\draw (11) to (7);
		\draw (3.center) to (4);
		\draw (8.center) to (1.center);
		\draw (0) to (13);
		\draw [in=180, out=-90] (14) to (5);
		\draw (10.center) to (14);
		\draw (15.center) to (11);
		\draw (13) to (16.center);
		\draw (17.center) to (5);
	\end{pgfonlayer}
\end{tikzpicture}
	\]
	Using this equation together with the assumption that $s$ is deterministic, we then compute
	
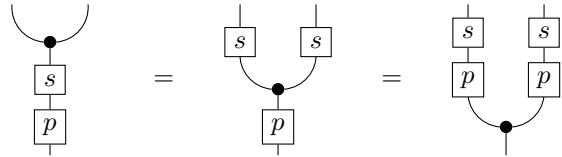
\[
		\begin{tikzpicture}
	\begin{pgfonlayer}{nodelayer}
		\node [style=bn] (0) at (-4, 1) {};
		\node [style=morphism] (1) at (-4, 0) {$s$};
		\node [style=morphism] (2) at (-4, -1.25) {$p$};
		\node [style=none] (3) at (-4, -2) {};
		\node [style=none] (4) at (-5, 2) {};
		\node [style=none] (5) at (-3, 2) {};
		\node [style=none] (6) at (-1, 0) {$=$};
		\node [style=bn] (7) at (2, -0.25) {};
		\node [style=morphism] (8) at (1, 1) {$s$};
		\node [style=morphism] (9) at (2, -1.25) {$p$};
		\node [style=none] (10) at (2, -2) {};
		\node [style=none] (11) at (1, 0.75) {};
		\node [style=none] (12) at (3, 0.75) {};
		\node [style=morphism] (13) at (3, 1) {$s$};
		\node [style=none] (14) at (1, 2) {};
		\node [style=none] (15) at (3, 2) {};
		\node [style=bn] (16) at (8, -1.25) {};
		\node [style=morphism] (17) at (7, 1.25) {$s$};
		\node [style=morphism] (18) at (7, 0) {$p$};
		\node [style=none] (19) at (8, -2) {};
		\node [style=none] (20) at (7, -0.25) {};
		\node [style=none] (21) at (9, -0.25) {};
		\node [style=morphism] (22) at (9, 1.25) {$s$};
		\node [style=none] (23) at (7, 2) {};
		\node [style=none] (24) at (9, 2) {};
		\node [style=none] (25) at (5, 0) {$=$};
		\node [style=morphism] (26) at (9, 0) {$p$};
	\end{pgfonlayer}
	\begin{pgfonlayer}{edgelayer}
		\draw [bend right=45] (4.center) to (0);
		\draw [bend right=45] (0) to (5.center);
		\draw (0) to (1);
		\draw (1) to (2);
		\draw (2) to (3.center);
		\draw [bend right=45] (11.center) to (7);
		\draw [bend right=45] (7) to (12.center);
		\draw (9) to (10.center);
		\draw (7) to (9);
		\draw (15.center) to (13);
		\draw (14.center) to (8);
		\draw (8) to (11.center);
		\draw (12.center) to (13);
		\draw [bend right=45] (20.center) to (16);
		\draw [bend right=45] (16) to (21.center);
		\draw (24.center) to (22);
		\draw (23.center) to (17);
		\draw (16) to (19.center);
		\draw (17) to (18);
		\draw (22) to (26);
	\end{pgfonlayer}
\end{tikzpicture}
	\]
	which is what was to be shown.
\end{proof}

Here is our abstract version of Kolmogorov's zero--one law.

\begin{theorem}[Abstract Kolmogorov zero--one law]
    \label{thm:kolmog}
    Suppose that $\cC$ is a Markov category, and that $(X_i)_{i \in J}$ is a family of objects with Kolmogorov product $X_J$ and finite marginalizations $\pi_F : X_J \to X_F$. Suppose that morphisms $p: A \to X_J$ and deterministic $s : X_J \to T$ satisfy the following:
    \begin{enumerate}
	    \item $p$ displays the conditional independence $\perp_i X_i \mid\mid A$.
	    \item For every finite $F \subseteq J$, the joint
    		
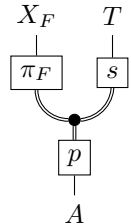
\[
	    		\begin{tikzpicture}
	\begin{pgfonlayer}{nodelayer}
		\node [style=bn] (0) at (0, 0) {};
		\node [style=morphism] (1) at (-1, 1.25) {$\pi_F$};
		\node [style=none] (2) at (0, -2) {};
		\node [style=none] (3) at (1, 1) {};
		\node [style=morphism] (4) at (0, -1) {$p$};
		\node [style=morphism] (5) at (1, 1.25) {$s$};
		\node [style=none] (6) at (-1, 2.25) {};
		\node [style=none] (7) at (1, 2.25) {};
		\node [style=none] (8) at (0, -2.5) {$A$};
		\node [style=none] (9) at (-1, 2.75) {$X_F$};
		\node [style=none] (10) at (1, 2.75) {$T$};
	\end{pgfonlayer}
	\begin{pgfonlayer}{edgelayer}
		\draw (7.center) to (5);
		\draw (5) to (3.center);
		\draw [style=double line] (3.center) to (5);
		\draw [style=double line, bend left=45] (3.center) to (0);
		\draw [style=double line, in=270, out=180] (0) to (1);
		\draw (1) to (6.center);
		\draw [style=double line] (0) to (4);
		\draw (4) to (2.center);
	\end{pgfonlayer}
\end{tikzpicture}
    		\]
    		displays the conditional independence $X_F \perp T \mid\mid A$.
    \end{enumerate}
Then the composite $sp : A \to T$ is deterministic as well.
\end{theorem}
Note that this result applies even in the case where $J$ is finite, in which case it amounts to \Cref{determinismlemma}.
\begin{proof}
	Assuming $\ast \not \in J$, let us write $X_\ast := T$, so that $X_J \otimes X_\ast$ is the Kolmogorov product of the family $(X_i)_{i \in J \,\sqcup\,\{\ast\}}$. It is clear from the definitions and the assumption that the morphism
    
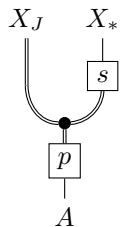
\[
	\begin{tikzpicture}
	\begin{pgfonlayer}{nodelayer}
		\node [style=bn] (0) at (0, 0) {};
		\node [style=none] (1) at (-1, 1) {};
		\node [style=none] (2) at (0, -2) {};
		\node [style=none] (3) at (1, 1) {};
		\node [style=morphism] (4) at (0, -1) {$p$};
		\node [style=morphism] (5) at (1, 1.25) {$s$};
		\node [style=none] (6) at (-1, 2.25) {};
		\node [style=none] (7) at (1, 2.25) {};
		\node [style=none] (8) at (0, -2.5) {$A$};
		\node [style=none] (9) at (-1, 2.75) {$X_J$};
		\node [style=none] (10) at (1, 2.75) {$X_\ast$};
	\end{pgfonlayer}
	\begin{pgfonlayer}{edgelayer}
		\draw (7.center) to (5);
		\draw (5) to (3.center);
		\draw [style=double line] (3.center) to (5);
		\draw [style=double line, bend left=45] (3.center) to (0);
		\draw [style=double line, in=270, out=180] (0) to (1.center);
		\draw [style=double line] (1.center) to (6.center);
		\draw [style=double line] (0) to (4);
		\draw (4) to (2.center);
	\end{pgfonlayer}
\end{tikzpicture}	    
    \]
    displays the conditional independence $\perp_{i \in J\,\sqcup\,\{\ast\}} X_i \mid\mid A$.
    By the infinite independence lemma, \cref{infindeplemma}, this means that it also displays the conditional independence $X_\ast \perp X_J \mid\mid A$.
    Now the determinism lemma implies exactly what we want, namely that the composite $sp : A \to T$ is deterministic.
\end{proof}

We now move on to our treatment of the Hewitt--Savage zero--one law. It applies to Kolmogorov products where all factors are copies of the same object. So for an object $X$ and the constant family $(X)_{i \in J}$, let us say that the corresponding Kolmogorov product $X_J$ is the \emph{Kolmogorov power} of $X$ by $J$, whenever it exists.

A \emph{finite permutation} of the set $J$ is a bijection $\sigma : J \to J$ which fixes all but finitely many elements. Suppose that $\sigma : J \to J$ is a finite permutation or more generally any injection. Then an application of the universal property of the Kolmogorov power shows that $\sigma$ induces a morphism $\hat{\sigma} : X_J \to X_J$ by applying the universal property of $X_J$ to the compatible family of composite morphisms
\[
	\begin{tikzcd}
		X_J \ar{r}{\pi_{\sigma(F)}} & X_{\sigma(F)} \ar{r}{\cong} & X_{F},
	\end{tikzcd}
\]
where the second morphism matches up the factors of the two finite products involved as prescribed by the bijection $\sigma|_F : F \stackrel{\cong}{\longrightarrow} \sigma(F)$. The relevant compatibility condition is straightforward to check. Moreover, $\hat{\sigma}$ is deterministic by \Cref{kolmprod_catdet}.

\begin{theorem}[Abstract Hewitt--Savage zero--one law]
    	\label{thm:hewsav}
	Suppose that $\cC$ is a causal Markov category (\Cref{causal_defn}). Let $X_J$ be a Kolmogorov power of an object $X \in \cC$ with respect to an infinite set $J$.
    	Suppose that morphisms $p: A \to X_J$ and deterministic $s : X_J \to T$ satisfy the following:
    	\begin{enumerate}
		\item $p$ displays the conditional independence $\perp_{i \in J} X_i \mid\mid A$.
		\item For every finite permutation $\sigma : J \to J$, we have $\hat{\sigma} p = p$ and $s \hat{\sigma} = s$. 
    	\end{enumerate}
	Then the composite $sp : A \to T$ is deterministic.
\end{theorem}

In contrast to \Cref{thm:kolmog}, we now have two relevant assumptions in addition to those on $p$ and $s$, namely that $\cC$ is causal and $J$ is infinite. Before getting to the proof, we first make a useful auxiliary observation.

\begin{lemma}
	\label{aseqlemma}
	Let $\cC$ be a causal Markov category. Let $p : A \to X$ and $f,g: X \to Y$ be morphisms in $\cC$ with $f$ deterministic and such that
	\[
		\begin{tikzpicture}
	\begin{pgfonlayer}{nodelayer}
		\node [style=bn] (0) at (-3.25, 0) {};
		\node [style=none] (1) at (-4.25, 1) {};
		\node [style=none] (2) at (-3.25, -2) {};
		\node [style=none] (3) at (-2.25, 1) {};
		\node [style=morphism] (4) at (-3.25, -1) {$p$};
		\node [style=morphism] (5) at (-2.25, 1.25) {$g$};
		\node [style=none] (6) at (-4.25, 2.25) {};
		\node [style=none] (7) at (-2.25, 2.25) {};
		\node [style=none] (11) at (0, 0) {$=$};
		\node [style=morphism] (12) at (-4.25, 1.25) {$f$};
		\node [style=bn] (13) at (3.25, 0) {};
		\node [style=none] (14) at (2.25, 1) {};
		\node [style=none] (15) at (3.25, -2) {};
		\node [style=none] (16) at (4.25, 1) {};
		\node [style=morphism] (17) at (3.25, -1) {$p$};
		\node [style=morphism] (18) at (4.25, 1.25) {$f$};
		\node [style=none] (19) at (2.25, 2.25) {};
		\node [style=none] (20) at (4.25, 2.25) {};
		\node [style=morphism] (21) at (2.25, 1.25) {$f$};
	\end{pgfonlayer}
	\begin{pgfonlayer}{edgelayer}
		\draw (7.center) to (5);
		\draw (5) to (3.center);
		\draw (3.center) to (5);
		\draw [bend left=45] (3.center) to (0);
		\draw [in=270, out=180] (0) to (1.center);
		\draw (0) to (4);
		\draw (4) to (2.center);
		\draw (6.center) to (12);
		\draw (12) to (1.center);
		\draw (20.center) to (18);
		\draw (18) to (16.center);
		\draw (16.center) to (18);
		\draw [bend left=45] (16.center) to (13);
		\draw [in=270, out=180] (13) to (14.center);
		\draw (13) to (17);
		\draw (17) to (15.center);
		\draw (19.center) to (21);
		\draw (21) to (14.center);
	\end{pgfonlayer}
\end{tikzpicture}
	\]
	Then $f =_{p\as} g$ (\Cref{defnasequal}).
\end{lemma}
\begin{proof}
    Consider the morphism
    \[
	    \begin{tikzpicture}
	\begin{pgfonlayer}{nodelayer}
		\node [style=bn] (0) at (0, 0) {};
		\node [style=none] (1) at (-1, 1) {};
		\node [style=none] (2) at (0, -2) {};
		\node [style=none] (3) at (1, 1) {};
		\node [style=morphism] (4) at (0, -1) {$p$};
		\node [style=morphism] (5) at (1, 1.25) {$g$};
		\node [style=bn] (6) at (-1, 2.25) {};
		\node [style=bn] (7) at (1, 2.25) {};
		\node [style=morphism] (8) at (-1, 1.25) {$f$};
		\node [style=none] (9) at (1.25, 3.5) {};
		\node [style=none] (10) at (-0.25, 3.75) {};
		\node [style=none] (11) at (-2, 3.75) {};
		\node [style=none] (12) at (1.25, 3.75) {};
		\node [style=none] (13) at (2.75, 3.75) {};
	\end{pgfonlayer}
	\begin{pgfonlayer}{edgelayer}
		\draw (7) to (5);
		\draw (5) to (3.center);
		\draw (3.center) to (5);
		\draw [bend left=45] (3.center) to (0);
		\draw [in=270, out=180] (0) to (1.center);
		\draw (0) to (4);
		\draw (4) to (2.center);
		\draw (6) to (8);
		\draw (8) to (1.center);
		\draw [in=-180, out=-90] (11.center) to (6);
		\draw [in=-90, out=0, looseness=0.50] (6) to (9.center);
		\draw [in=180, out=-90] (10.center) to (7);
		\draw (12.center) to (9.center);
		\draw [in=0, out=-90] (13.center) to (7);
	\end{pgfonlayer}
\end{tikzpicture}
    \]
    Applying the assumed equation together with determinism of $f$ shows that marginalizing over the first output gives the same morphism $A \to Y \otimes Y \otimes Y$ as marginalizing over the second output. Therefore, applying the causality property of \Cref{causal_defn} with $h_1$ and $h_2$ the two marginalizations, we can conclude the equality
    \[
	\begin{tikzpicture}
	\begin{pgfonlayer}{nodelayer}
		\node [style=bn] (0) at (-5, -1) {};
		\node [style=none] (1) at (-7, 0) {};
		\node [style=none] (2) at (-5, -3) {};
		\node [style=none] (3) at (-5, 0) {};
		\node [style=morphism] (4) at (-5, -2) {$p$};
		\node [style=morphism] (5) at (-5, 0.25) {$g$};
		\node [style=bn] (6) at (-7, 1.25) {};
		\node [style=bn] (7) at (-5, 1.25) {};
		\node [style=morphism] (8) at (-7, 0.25) {$f$};
		\node [style=none] (9) at (-4.5, 2.75) {};
		\node [style=none] (10) at (-6.25, 3) {};
		\node [style=bn] (11) at (-8, 2.25) {};
		\node [style=none] (12) at (-4.5, 3) {};
		\node [style=none] (13) at (-2, 3) {};
		\node [style=none] (14) at (-2, 0.5) {};
		\node [style=bn] (15) at (5, -1) {};
		\node [style=none] (16) at (3, 0) {};
		\node [style=none] (17) at (5, -3) {};
		\node [style=none] (18) at (5, 0) {};
		\node [style=morphism] (19) at (5, -2) {$p$};
		\node [style=morphism] (20) at (5, 0.25) {$g$};
		\node [style=bn] (21) at (3, 1.25) {};
		\node [style=bn] (22) at (5, 1.25) {};
		\node [style=morphism] (23) at (3, 0.25) {$f$};
		\node [style=none] (24) at (5.5, 2.75) {};
		\node [style=bn] (25) at (4, 2.25) {};
		\node [style=none] (26) at (2, 3) {};
		\node [style=none] (27) at (5.5, 3) {};
		\node [style=none] (28) at (8, 3) {};
		\node [style=none] (29) at (8, 0.5) {};
		\node [style=none] (30) at (0, 0) {$=$};
		\node [style=none] (31) at (6.75, 3) {};
		\node [style=none] (32) at (-3.25, 3) {};
	\end{pgfonlayer}
	\begin{pgfonlayer}{edgelayer}
		\draw (7) to (5);
		\draw (5) to (3.center);
		\draw (3.center) to (5);
		\draw [in=270, out=180] (0) to (1.center);
		\draw (0) to (4);
		\draw (4) to (2.center);
		\draw (6) to (8);
		\draw (8) to (1.center);
		\draw [in=-180, out=-90] (11) to (6);
		\draw [in=-90, out=0, looseness=0.50] (6) to (9.center);
		\draw [in=180, out=-90] (10.center) to (7);
		\draw (12.center) to (9.center);
		\draw (3.center) to (0);
		\draw [in=0, out=-90] (14.center) to (0);
		\draw (14.center) to (13.center);
		\draw (22) to (20);
		\draw (20) to (18.center);
		\draw (18.center) to (20);
		\draw [in=270, out=180] (15) to (16.center);
		\draw (15) to (19);
		\draw (19) to (17.center);
		\draw (21) to (23);
		\draw (23) to (16.center);
		\draw [in=-180, out=-90] (26.center) to (21);
		\draw [in=-90, out=0, looseness=0.50] (21) to (24.center);
		\draw [in=-180, out=-90] (25) to (22);
		\draw (27.center) to (24.center);
		\draw (18.center) to (15);
		\draw [in=0, out=-90] (29.center) to (15);
		\draw (29.center) to (28.center);
		\draw [in=0, out=-90] (32.center) to (7);
		\draw [in=-90, out=0] (22) to (31.center);
	\end{pgfonlayer}
\end{tikzpicture}
    \]
    Marginalizing this further over the second and third outputs and truncating the redundant pieces proves the desired equation.
\end{proof}

\begin{proof}[Proof of \Cref{thm:hewsav}]
	Letting a finite permutation $\sigma : J \to J$ act on the $X_J$ part of the joint of $X_J$ and $T$ gives
	\begin{equation}
		\label{strong_invariance}
		\begin{tikzpicture}
	\begin{pgfonlayer}{nodelayer}
		\node [style=bn] (0) at (-4, 0) {};
		\node [style=none] (1) at (-5, 1) {};
		\node [style=none] (2) at (-4, -2) {};
		\node [style=none] (3) at (-3, 1) {};
		\node [style=morphism] (4) at (-4, -1) {$p$};
		\node [style=morphism] (5) at (-3, 1.25) {$s$};
		\node [style=none] (6) at (-5, 2.25) {};
		\node [style=none] (7) at (-3, 2.25) {};
		\node [style=morphism] (8) at (-5, 1.25) {$\hat{\sigma}$};
		\node [style=none] (9) at (-1, 0) {$=$};
		\node [style=bn] (10) at (2, -0.5) {};
		\node [style=none] (11) at (1, 0.5) {};
		\node [style=none] (12) at (2, -2.5) {};
		\node [style=none] (13) at (3, 0.5) {};
		\node [style=morphism] (14) at (2, -1.5) {$p$};
		\node [style=morphism] (15) at (3, 2) {$s$};
		\node [style=none] (16) at (1, 2.75) {};
		\node [style=none] (17) at (3, 2.75) {};
		\node [style=morphism] (18) at (1, 0.75) {$\hat{\sigma}$};
		\node [style=morphism] (19) at (3, 0.75) {$\hat{\sigma}$};
		\node [style=none] (20) at (5, 0) {$=$};
		\node [style=bn] (21) at (8, 0.75) {};
		\node [style=none] (22) at (7, 1.75) {};
		\node [style=none] (23) at (8, -2.5) {};
		\node [style=none] (24) at (9, 1.75) {};
		\node [style=morphism] (25) at (8, -1.5) {$p$};
		\node [style=morphism] (26) at (9, 2) {$s$};
		\node [style=none] (27) at (7, 2.75) {};
		\node [style=none] (28) at (9, 2.75) {};
		\node [style=morphism] (29) at (8, -0.25) {$\hat{\sigma}$};
		\node [style=bn] (30) at (14, 0) {};
		\node [style=none] (31) at (13, 1) {};
		\node [style=none] (32) at (14, -2) {};
		\node [style=none] (33) at (15, 1) {};
		\node [style=morphism] (34) at (14, -1) {$p$};
		\node [style=morphism] (35) at (15, 1.25) {$s$};
		\node [style=none] (36) at (13, 2.25) {};
		\node [style=none] (37) at (15, 2.25) {};
		\node [style=none] (39) at (11, 0) {$=$};
	\end{pgfonlayer}
	\begin{pgfonlayer}{edgelayer}
		\draw (7.center) to (5);
		\draw (5) to (3.center);
		\draw (3.center) to (5);
		\draw [style=double line, bend left=45] (3.center) to (0);
		\draw [style=double line, in=270, out=180] (0) to (1.center);
		\draw [style=double line] (0) to (4);
		\draw (4) to (2.center);
		\draw [style=double line] (6.center) to (8);
		\draw (8) to (1.center);
		\draw (17.center) to (15);
		\draw [style=double line, bend left=45] (13.center) to (10);
		\draw [style=double line, in=270, out=180] (10) to (11.center);
		\draw [style=double line] (10) to (14);
		\draw (14) to (12.center);
		\draw [style=double line] (16.center) to (18);
		\draw (18) to (11.center);
		\draw [style=double line] (13.center) to (19);
		\draw [style=double line, in=270, out=90] (19) to (15);
		\draw (28.center) to (26);
		\draw [style=double line, bend left=45] (24.center) to (21);
		\draw [style=double line, in=270, out=180] (21) to (22.center);
		\draw (25) to (23.center);
		\draw [style=double line] (27.center) to (22.center);
		\draw [style=double line] (26) to (24.center);
		\draw [style=double line] (21) to (29);
		\draw [style=double line] (29) to (25);
		\draw (37.center) to (35);
		\draw (35) to (33.center);
		\draw (33.center) to (35);
		\draw [style=double line, bend left=45] (33.center) to (30);
		\draw [style=double line, in=270, out=180] (30) to (31.center);
		\draw [style=double line] (30) to (34);
		\draw (34) to (32.center);
		\draw [style=double line] (36.center) to (31.center);
	\end{pgfonlayer}
\end{tikzpicture}
	\end{equation}
	so that the joint given by the right-hand side is invariant. Here, the first and last equation hold by the invariance assumption, and the second because $\hat{\sigma}$ is deterministic.

	We next argue that this equation still holds if $\sigma$ is merely an injection. By the universal property of $X_J \otimes T$ as the infinite tensor product of the family $(X_i)_{i \in J}$ together with $T$, it is enough to prove \Cref{strong_invariance} upon marginalization by $\pi_F \otimes \id$ for all finite $F \subseteq J$. Since we can find a finite permutation $\sigma' : J \to J$ such that $\sigma'|_F = \sigma|_F$ for any given $F$, the claim follows from the finite permutation case.
    
	Consider the morphism given by composing \eqref{strong_invariance} with $s$ on the first output. Since these are the same for all injections $\sigma$, \Cref{aseqlemma} tells us that the morphisms $s \hat{\sigma}$ are all $p$-a.s.~equal to $s$, for every injection $\sigma : J \to J$.

	Now let $J = J_1 \,\sqcup\, J_2$ be a decomposition of $J$ into two disjoint subsets having the same cardinality as $J$; this is the step which relies on $J$ being infinite. Let $\tau_1, \tau_2: J \to J$ be injections with images $J_1$ and $J_2$, respectively. Then we claim that
	\begin{equation}
		\label{hewsav2}
		\begin{tikzpicture}
	\begin{pgfonlayer}{nodelayer}
		\node [style=bn] (0) at (-3.25, 0) {};
		\node [style=bn] (1) at (3.25, -1.25) {};
		\node [style=none] (2) at (-3.25, -2) {};
		\node [style=morphism] (3) at (-3.25, -1) {$p$};
		\node [style=morphism] (4) at (-4.25, 1.25) {$\hat{\tau}_1$};
		\node [style=morphism] (5) at (-2.25, 1.25) {$\hat{\tau}_2$};
		\node [style=none] (6) at (-4.25, 2.5) {};
		\node [style=none] (7) at (-2.25, 2.5) {};
		\node [style=none] (8) at (-4.25, 1) {};
		\node [style=none] (9) at (-2.25, 1) {};
		\node [style=none] (10) at (3.25, -2) {};
		\node [style=none] (11) at (2.25, -0.25) {};
		\node [style=none] (12) at (4.25, -0.25) {};
		\node [style=morphism] (13) at (2.25, 0) {$p$};
		\node [style=morphism] (14) at (4.25, 0) {$p$};
		\node [style=morphism] (15) at (2.25, 1.5) {$\hat{\tau}_1$};
		\node [style=morphism] (16) at (4.25, 1.5) {$\hat{\tau}_2$};
		\node [style=none] (17) at (2.25, 2.5) {};
		\node [style=none] (18) at (4.25, 2.5) {};
		\node [style=none] (19) at (0, 0) {$=$};
	\end{pgfonlayer}
	\begin{pgfonlayer}{edgelayer}
		\draw (9.center) to (5);
		\draw (8.center) to (4);
		\draw [style=double line] (4) to (6.center);
		\draw [style=double line] (7.center) to (9.center);
		\draw [style=double line, bend right=45] (8.center) to (0);
		\draw [style=double line, bend right=45] (0) to (9.center);
		\draw [style=double line] (0) to (3);
		\draw (3) to (2.center);
		\draw [bend right=45] (11.center) to (1);
		\draw (1) to (10.center);
		\draw [bend right=45] (1) to (12.center);
		\draw [style=double line] (12.center) to (14);
		\draw [style=double line] (11.center) to (13);
		\draw [style=double line] (13) to (15);
		\draw [style=double line] (15) to (17.center);
		\draw [style=double line] (18.center) to (16);
		\draw [style=double line] (16) to (14);
	\end{pgfonlayer}
\end{tikzpicture}
	\end{equation}
	By \Cref{two_infproducts}, it is enough to prove this upon postcomposing with two finite marginalizations $\pi_{F_1}$ and $\pi_{F_2}$. But then using the fact that $\tau_1$ and $\tau_2$ have disjoint image, this follows from the independence assumption, which has
	\[
		\begin{tikzpicture}
	\begin{pgfonlayer}{nodelayer}
		\node [style=bn] (0) at (-4.5, 0) {};
		\node [style=bn] (1) at (4.5, -1.25) {};
		\node [style=none] (2) at (-4.5, -2) {};
		\node [style=morphism] (3) at (-4.5, -1) {$p$};
		\node [style=morphism] (4) at (-6.25, 1.25) {$\pi_{\tau_1(F_1)}$};
		\node [style=morphism] (5) at (-2.75, 1.25) {$\pi_{\tau_2(F_2)}$};
		\node [style=none] (6) at (-6.25, 2.5) {};
		\node [style=none] (7) at (-2.75, 2.5) {};
		\node [style=none] (8) at (-6.25, 1) {};
		\node [style=none] (9) at (-2.75, 1) {};
		\node [style=none] (10) at (4.5, -2) {};
		\node [style=none] (11) at (2.75, -0.25) {};
		\node [style=none] (12) at (6.25, -0.25) {};
		\node [style=morphism] (13) at (2.75, 0) {$p$};
		\node [style=morphism] (14) at (6.25, 0) {$p$};
		\node [style=morphism] (15) at (2.75, 1.5) {$\pi_{\tau_1(F_1)}$};
		\node [style=morphism] (16) at (6.25, 1.5) {$\pi_{\tau_2(F_2)}$};
		\node [style=none] (17) at (2.75, 2.5) {};
		\node [style=none] (18) at (6.25, 2.5) {};
		\node [style=none] (19) at (0, 0) {$=$};
	\end{pgfonlayer}
	\begin{pgfonlayer}{edgelayer}
		\draw (9.center) to (5);
		\draw (8.center) to (4);
		\draw (4) to (6.center);
		\draw (7.center) to (9.center);
		\draw [style=double line, in=-180, out=-90] (8.center) to (0);
		\draw [style=double line, in=-90, out=0] (0) to (9.center);
		\draw [style=double line] (0) to (3);
		\draw (3) to (2.center);
		\draw [in=-180, out=-90] (11.center) to (1);
		\draw (1) to (10.center);
		\draw [in=-90, out=0] (1) to (12.center);
		\draw [style=double line] (12.center) to (14);
		\draw [style=double line] (11.center) to (13);
		\draw [style=double line] (13) to (15);
		\draw (15) to (17.center);
		\draw (18.center) to (16);
		\draw [style=double line] (16) to (14);
	\end{pgfonlayer}
\end{tikzpicture}
	\]
	as a special case.
    
	Putting things together, we can now compute
	\[
		\begin{tikzpicture}
	\begin{pgfonlayer}{nodelayer}
		\node [style=bn] (0) at (-12, 1) {};
		\node [style=morphism] (1) at (-12, 0) {$s$};
		\node [style=morphism] (2) at (-12, -1.25) {$p$};
		\node [style=none] (3) at (-12, -2) {};
		\node [style=none] (4) at (-13, 2) {};
		\node [style=none] (5) at (-11, 2) {};
		\node [style=none] (6) at (-9, 0) {$=$};
		\node [style=bn] (8) at (-6, 0) {};
		\node [style=morphism] (9) at (-6, -1) {$p$};
		\node [style=none] (10) at (-6, -2) {};
		\node [style=none] (11) at (-7, 2) {};
		\node [style=none] (12) at (-5, 2) {};
		\node [style=morphism] (13) at (-7, 1.25) {$s$};
		\node [style=morphism] (14) at (-5, 1.25) {$s$};
		\node [style=none] (15) at (-5, 1) {};
		\node [style=none] (16) at (-7, 1) {};
		\node [style=none] (17) at (-3, 0) {$=$};
		\node [style=bn] (18) at (0, -0.75) {};
		\node [style=morphism] (19) at (0, -1.75) {$p$};
		\node [style=none] (20) at (0, -2.75) {};
		\node [style=none] (21) at (-1, 2.5) {};
		\node [style=none] (22) at (1, 2.5) {};
		\node [style=morphism] (23) at (-1, 1.75) {$s$};
		\node [style=morphism] (24) at (1, 1.75) {$s$};
		\node [style=none] (25) at (1, 0.25) {};
		\node [style=none] (26) at (-1, 0.25) {};
		\node [style=morphism] (27) at (-1, 0.5) {$\hat{\tau_1}$};
		\node [style=morphism] (28) at (1, 0.5) {$\hat{\tau_2}$};
		\node [style=bn] (29) at (6, -2) {};
		\node [style=none] (30) at (6, -2.75) {};
		\node [style=none] (31) at (5, -1) {};
		\node [style=none] (32) at (7, -1) {};
		\node [style=morphism] (33) at (5, -0.75) {$p$};
		\node [style=morphism] (34) at (7, -0.75) {$p$};
		\node [style=morphism] (35) at (5, 0.5) {$\hat{\tau}_1$};
		\node [style=morphism] (36) at (7, 0.5) {$\hat{\tau}_2$};
		\node [style=morphism] (37) at (5, 1.75) {$s$};
		\node [style=morphism] (38) at (7, 1.75) {$s$};
		\node [style=none] (39) at (3, 0) {$\stackrel{\eqref{hewsav2}}{=}$};
		\node [style=none] (40) at (5, 2.5) {};
		\node [style=none] (41) at (7, 2.5) {};
		\node [style=none] (42) at (9, 0) {$=$};
		\node [style=bn] (43) at (12, -1.25) {};
		\node [style=none] (44) at (12, -2) {};
		\node [style=none] (45) at (11, -0.25) {};
		\node [style=none] (46) at (13, -0.25) {};
		\node [style=morphism] (47) at (11, 0) {$p$};
		\node [style=morphism] (48) at (13, 0) {$p$};
		\node [style=morphism] (51) at (11, 1.25) {$s$};
		\node [style=morphism] (52) at (13, 1.25) {$s$};
		\node [style=none] (53) at (11, 2) {};
		\node [style=none] (54) at (13, 2) {};
	\end{pgfonlayer}
	\begin{pgfonlayer}{edgelayer}
		\draw [bend right=45] (4.center) to (0);
		\draw [bend right=45] (0) to (5.center);
		\draw (0) to (1);
		\draw (2) to (3.center);
		\draw [style=double line] (1) to (2);
		\draw (9) to (10.center);
		\draw [style=double line] (8) to (9);
		\draw [style=double line, bend left=45] (15.center) to (8);
		\draw [style=double line, bend right=315] (8) to (16.center);
		\draw [style=double line] (16.center) to (13);
		\draw (13) to (11.center);
		\draw (12.center) to (14);
		\draw [style=double line] (14) to (15.center);
		\draw (19) to (20.center);
		\draw [style=double line] (18) to (19);
		\draw [style=double line, bend left=45] (25.center) to (18);
		\draw [style=double line, bend right=315] (18) to (26.center);
		\draw (21.center) to (23);
		\draw (22.center) to (24);
		\draw [style=double line] (23) to (27);
		\draw [style=double line] (24) to (28);
		\draw [bend right=45] (31.center) to (29);
		\draw (29) to (30.center);
		\draw [bend right=45] (29) to (32.center);
		\draw [style=double line] (32.center) to (34);
		\draw [style=double line] (31.center) to (33);
		\draw [style=double line] (33) to (35);
		\draw [style=double line] (35) to (37);
		\draw [style=double line] (38) to (36);
		\draw [style=double line] (36) to (34);
		\draw (40.center) to (37);
		\draw (41.center) to (38);
		\draw [bend right=45] (45.center) to (43);
		\draw (43) to (44.center);
		\draw [bend right=45] (43) to (46.center);
		\draw [style=double line] (46.center) to (48);
		\draw [style=double line] (45.center) to (47);
		\draw (53.center) to (51);
		\draw (54.center) to (52);
		\draw [style=double line] (51) to (47);
		\draw [style=double line] (52) to (48);
	\end{pgfonlayer}
\end{tikzpicture}
	\]
	which is what was to be shown.
\end{proof}

\section{Examples}
\label{examples}

We now show how our results specialize to the standard Kolmogorov and Hewitt--Savage zero--one laws in the case of standard Borel spaces. Further down we will give another application to continuous maps between topological spaces.

\begin{corollary}[Kolmogorov zero--one law]
	Suppose $\Omega$ is a standard Borel space with a probability measure $P$ and that $\left(f_i: \Omega \to X_i\right)_{i \in \bN}$ is a sequence of independent random variables taking values in standard Borel spaces $X_i$. Suppose that $T \subseteq \Omega$ is in the coarsest $\sigma$-algebra which makes the $f_i$ measurable, and independent of any finite subset of the $f_i$. Then $P(T) \in \{0,1\}$.
\end{corollary}
\begin{proof}
	We want to apply \Cref{thm:kolmog} with $\cC = \BorelStoch$ and $J = \bN$ and $A = I$, the one-element measurable space. Then $P : I \to \Omega$. We write $f_\bN : \Omega \to X_\bN$ for the induced map to the product measurable space. In $\Omega$, the coarsest $\sigma$-algebra which makes the $f_i$ measurable coincides with the $\sigma$-algebra consisting of the sets of the form $f_\bN^{-1}(S)$ for measurable $S \subseteq X_J$. Hence the assumption implies that there is measurable $S \subseteq X_J$ such that $T = f_\bN^{-1}(S)$. Then we apply \Cref{thm:kolmog} with $s : X_J \to \{0,1\}$ the indicator function of $S$, which is a measurable map and therefore can also be regarded as a morphism in $\BorelStoch$, and with $p = f_\bN P$. We therefore obtain that $sp : I \to \{0,1\}$ is deterministic. But since $sp$ corresponds to the probability measure on $\{0,1\}$ with weight $P(f_\bN^{-1}(S)) = P(T)$ on $\{1\}$ and the complementary weight $1 - P(T)$ on $0$, we conclude that indeed $P(T) \in \{0,1\}$.
\end{proof}

We note that the standard Kolmogorov zero--one law holds for any measurable
spaces, not just for standard Borel spaces, and hence this is a strictly less
general statement. As discussed in \cref{infprods_stoch}, it seems unlikely that
the general statement can be proven by applying our machinery to the category
$\Stoch$ of measurable spaces and Markov kernels. This leads to the following problem:

\begin{problem}
	Is there a Markov category $\cC$ with countable Kolmogorov products such that \Cref{thm:kolmog} specializes to the standard Kolmogorov zero--one law, in its general form applicable to all measurable spaces?
  \end{problem}

\begin{corollary}[Hewitt-savage zero--one law]
    Suppose $\Omega$ is a standard Borel space with probability measure $P$, and let $\left(f_i: \Omega \to X\right)_{i\in \bN}$ be a sequence of independent, identically distributed random variables taking values in a standard Borel space $X$.
    Suppose $S \subseteq \prod_{i \in \bN} X = X_\bN$ is measurable in the product $\sigma$-algebra, and suppose moreover that, for each finite permutation $\sigma: \bN \to \bN$, we have $\hat{\sigma}(S) = S$ (in other words, $S$ is invariant under finite permutations of the variables).
    Then $P(f_\bN\inv(S)) \in \{0,1\}$, where $f_\bN$ again denotes the induced map $\Omega \to X_\bN$.
  \end{corollary}

\begin{proof}
    The proof is similar to the previous proof. This time we want to apply \Cref{thm:hewsav} with $\cC = \BorelStoch$, $J = \bN$ and $A = I$.
    Again, we have $P: I \to \Omega$.
    As before, we let $s: X_\bN \to \{0,1\}$ be the indicator function of $S$, and we let $p = f_\bN P$.
    The i.i.d.~assumption on the $f_i$ implies that $\hat{\sigma}p = p$, and the assumption on $S$ implies $s\hat{\sigma} = s$ for finite permutations $\sigma$.
    Hence we can apply \Cref{thm:hewsav}. This again shows that the map $sp: I \to \{0,1\}$ is deterministic, which, just as before, means $P(f_\bN\inv(S)) \in \{0,1\}$.
\end{proof} 

By interpreting these results in different Markov categories, we can obtain other results.
In particular, the following example is of a similar flavour as $\SetMulti$, but improves on its lack of infinite tensor products (\Cref{setmulti}).

With $\Top$ the category of topological spaces and continuous maps, we consider the \emph{lower Vietoris monad} or \emph{hyperspace monad} $H$ on $\Top$, introduced by Schalk for $T_0$ spaces~\cite{schalk} and investigated generally in~\cite[Section~2]{hyperspace}. Let $\Kl(H)$ be its Kleisli category. Using the symmetric monoidal structure of $H$ guaranteed by~\cite[Corollary~2.53]{hyperspace}, the general construction of Markov categories from symmetric monoidal monads of \cite[Proposition~3.1]{markov_cats} turns $\Kl(H)$ into a Markov category. Concretely, this category has the following structure:
    \begin{enumerate}
        \item The objects are topological spaces.
	\item A morphism $f: X \to Y$ is a continuous function $X \to HY$, where $HY$ denotes the set of closed subsets of the space $Y$, equipped with the topology generated by subbasic opens of the form
		\[
			\mathrm{Hit}(U) := \{C \subseteq Y \text{ closed} \mid C \cap U \neq \emptyset\},
		\]
		where $U \subseteq Y$ ranges over all opens.
	\item Composition is defined as in $\SetMulti$ via~\eqref{relational_composition}, but taking the closure of the union in addition.
        \item The monoidal product is the Cartesian product of topological spaces, and on morphisms
		\[
			(f \tensor g)(a,b) := f(a) \times g(b),
		\]
		exactly as in $\SetMulti$.
        \item The comultiplication is defined by $\cop{X}(x) = \cl(\{(x,x)\})$
    \end{enumerate}
\begin{proposition}
    \label{infprods_klH}
    $\Kl(H)$ has Kolmogorov products of any cardinality, given by the usual infinite product of topological spaces with the product topology.
\end{proposition}
\begin{proof}
    First we observe that $H(\prod_{i\in J} X_i)$ is the cofiltered limit of the spaces
    $H(\prod_{i\in F} X_i)$ for $F \subseteq J$ finite.
    To see this, note that the marginalizations induce a continuous map
    	\[
		\alpha \: : \: H\!\left(\prod_{i\in J} X_i\right) \longrightarrow \lim_{F \subseteq J} \, H\!\left(\prod_{i \in J} X_i\right)
	\]
	We prove that this is a homeomorphism, showing first that it is an injection. Thus we start with distinct closed subsets $C, D \subseteq \prod X_i$. We assume without loss of generality that $C \not\subseteq D$, and choose $x \in C \setminus D$.
    Then let $U \ni x$ be an open neighborhood of $x$ not intersecting $D$.
    By making $U$ smaller if necessary, we can assume without loss of generality that $U$ is a basic open in the product topology, meaning that there is a finite $F$ so that
    \begin{equation}
	    \label{basicU}
		y = (y_i) \in U \quad\Longleftrightarrow\quad y_i \in \pi_i(U) \:\; \forall i \in F.
    \end{equation}
    Then $\pi_F(x) \in \pi_F(U)$, which is disjoint from $\pi_F(D)$. Hence $\alpha(C) \neq \alpha(D)$.

    We show that $\alpha$ is also surjective. Given a family $(C_F) \in \lim_F H\!\left(\prod_{i \in F} X_i\right)$, simply set
	\[
		C := \left\{x \in \prod_{j\in J} X_j \:\Bigg|\: \pi_F(x) \in C_F \:\: \forall F\right\}
	\]
	This is an intersection of closed subsets, hence closed, and the finite marginalizations clearly send it to the given family $(C_F)$.

	Finally, we must verify that $\alpha$ is an open map.
	It suffices to show that it carries a subbasic open set $\mathrm{Hit}(U) = \{C \subseteq \prod X_i \text{ closed} \mid C \cap U \neq \emptyset\}$ to an open subset of the limit. For a given point $C \in \mathrm{Hit}(U)$, it is enough to find an open neighborhood contained in $\mathrm{Hit}(U)$. 
	But then again by choosing smaller $U$ if necessary, so that $C \cap U \neq \emptyset$ still holds, we can assume~\eqref{basicU} without loss of generality.
    Then
	\[
		\alpha(A_U) = \pi_F\inv\left(\left\{D \subseteq \prod_{j \in F} X_F \text{ closed} \:\Bigg|\: D \cap \pi_F(U) \neq \emptyset\right\}\right),
	\]
    which is indeed open in the limit.

    It follows that $\prod_i X_i$ is an infinite tensor product in $\Kl(H)$; the preservation condition is easily verified since adding another factor to this product again gives the infinite product. The Kolmogorov product condition is similarly immediate to verify.
\end{proof}

We can therefore instantiate our \Cref{thm:kolmog} in $\Kl(H)$, which gives the following result as a special case.

\begin{corollary}
    Let $(X_i)_{i \in J}$ be a family of topological spaces, $Y$ a Hausdorff space, and let $f: \prod_i X_i \to Y$ be a continuous function
    which is independent of any finite subset of the input.
    Then $f$ is constant.
    \label{vietoris_kolmogorov}
\end{corollary}
\begin{proof}
	These results follow from \Cref{thm:kolmog} applied to $\Kl(H)$,
    using the morphism $p : I \to X_J$ corresponding to the trivial closed subset $X_J \subseteq X_J$.
	The conclusion is that the closure of its image, call it $A$, satisfies
	\[
		\cl(\{(a,a) \mid a \in A\}) = A \times A
	\]
	as subsets of $Y \times Y$.
    But since the diagonal is closed in $Y$, this clearly implies that $A$ must be a singleton.
\end{proof}

However, we cannot instantiate the Hewitt--Savage zero--one law in $\Kl(H)$, since this Markov category is in fact not causal, as the following example shows.

\begin{example}
	Let $A = \{*\}$ and $X = Y = Z = [0,1]$ with the usual topology.
    Let $f: A \to HX$ send $*$ to $X$, and define
    \[g(x) :=  \begin{cases} [0,1/2] & x < \frac{1}{3}\\ \{1/2\} & \frac{1}{3} \leq x \leq \frac{2}{3}\\ [1/2,1] & x > \frac{1}{3}\end{cases}\]
    \[
    	h_1(y) := [0,1], \qquad h_2(y) = \begin{cases}[0,1] & y \neq \frac{1}{2}\\ \{0\} &  y = \frac{1}{2}\end{cases}
    \]
    Using the definition of the hyperspace topology, it is straightforward to see that these maps are continuous. To check the hypothesis of the statement of causality, we must verify that the resulting two maps $A \to Y \times Z$ agree; these are just closed subsets of $[0,1]^2$. From the definition of the composition, we see that any point $(y,z)$ with $y \in g(x)$ for some $x$ and $z \in h_1(y)$ must lie in the set corresponding to $h_1$,
and similarly for $h_2$.
Hence the set corresponding to $h_1$ contains all points, while the set
corresponding to $h_2$ contains as a subset all points $(y,z)$ with $y\neq 1/2$,
as well as $(1/2,0)$.
But the closure of $\{(y,z) \mid y \neq 1/2\}$ in $[0,1]^2$ is the whole space, so that both sets are indeed equal to the whole space.

Hence the hypothesis is satisfied.
But one can verify that the point $(x,y,z) = (1/2,1/2,1)$ is in the subset of $X \times Y \times Z$ corresponding to the $h_1$ side of \eqref{causal2}, but not so for the $h_2$ side.
\end{example}

We end with an open problem.

\begin{problem}
	Find an interesting causal Markov category which has all Kolmogorov products.
\end{problem}

Strangely enough, the only examples that we know of so far are of two kinds: first, cartesian monoidal categories (meaning that $\cC = \cC_{\det}$) with all products; and second, the finitary complete monoids of \Cref{poset_case} which are monoidal posets. For the former, the universal property of the monoidal structure implies that the causality axiom holds. For the latter, it holds trivially since any two parallel morphisms are equal.

\bibliographystyle{plain}
\bibliography{categorical-zero-one}

\end{document}